\documentclass[sn-mathphys]{sn-jnl}
\jyear{2021}
\usepackage{graphicx}
\usepackage{epstopdf}
\theoremstyle{thmstyleone}%
\newtheorem{theorem}{Theorem}[section]
\newtheorem{proposition}[theorem]{Proposition}%
\newtheorem{lemma}[theorem]{Lemma}
\newtheorem{assumption}[theorem]{Assumption}

\newtheorem{corollary}[theorem]{Corollary}
\newtheorem{remark}[theorem]{Remark}%

\theoremstyle{thmstylethree}%

\begin{document}

\title[\tiny First order strong approximation of Ait--Sahalia-type interest rate model with Poisson jumps]{First order strong approximation of Ait--Sahalia-type interest rate model with Poisson jumps}


\author[1]{\fnm{Ziyi} \sur{Lei}}\email{csu\_ziyilei@csu.edu.cn}
\author*[1]{\fnm{Siqing} \sur{Gan}}\email{sqgan@csu.edu.cn}
\author[1]{\fnm{Jing} \sur{Liu}}\email{1014913478@qq.com}

\affil[1]{\orgdiv{School of Mathematics and Statistics, HNP-LAMA}, \orgname{Central South University}, \orgaddress{\city{Changsha}, \postcode{410083}, \country{China}}}
 \footnotetext{This work was supported by Natural Science Foundation of China (11971488).}

\abstract{For Ait--Sahalia-type interest rate model with Poisson jumps, we are interested in strong convergence of a novel time-stepping method, called transformed jump-adapted backward Euler method (TJABEM). Under certain hypotheses, the considered model takes values in positive domain. It is shown that the TJABEM can preserve the domain of the underlying problem. Furthermore, the first-order convergence rate of the TJABEM is recovered with respect to a $L^{p}$-error criterion. Numerical experiments are finally given to illustrate the theoretical results.}



\maketitle

\section{Introduction}\label{sec1}
As mentioned in \cite{Platen2010Numerical}, in financial and actuarial modeling and other areas of application, jump diffusions are often used to describe the dynamics of various state variables, like asset prices, credit ratings and interest rates. The jump component can capture event-driven uncertainties, such as corporate defaults, operational failures or insured events. Therefore, the research directions on stochastic differential equations (SDEs) with jumps have been warmly pursued, and many excellent research results have been obtained in numerical analysis due to no available exact solutions to these models (see \cite{Higham2005Numerical,Higham2007Strong,Wang2010Compensated,Deng2019Truncated,Chen2019Mean,Chen2020Convergence,Ren2020Compensated,Zhao2021On}). Meanwhile, we would like to mention that, as a nonlinear SDE, Ait--Sahalia-type interest rate model plays an important role in mathematical finance. With the aid of a series of studies by scholars, mainly including Ait--Sahalia \cite{Ait1996Testing}, we've learned that Ait--Sahalia-type model captures well of the dynamics of the spot rate in the research of several continuous-time models for interest rates. Moreover, it is widely used to volatility and other financial quantities besides interest rate now. The more detailed description of this model can be found in the literatures \cite{Deng2019Generalized,Szpruch2011Numerical,Hong2010Modeling,Jiang2017Proporty,Jin2016Ergodicity}. It can be seen that the numerical issues arising from Ait--Sahalia-type interest rate model are worth to analyze. \par
In this paper, we consider the Ait--Sahalia-type interest rate model with Poisson jumps of the form
\begin{eqnarray}\label{1.1}
	dX_t&=&(\alpha_{-1}X_{t-}^{-1}-\alpha_0+\alpha_1 X_{t-}-\alpha_2 X_{t-}^\gamma)dt\nonumber\\
	&&+\alpha_3 X_{t-}^\rho dW_{t} +h(X_{t-})dN_{t},\ \  t>0, \\ X_{0}&=&x_{0},\nonumber
\end{eqnarray}
where parameters $\alpha_{-1},\alpha_0,\alpha_1,\alpha_2,\alpha_3,\gamma,\rho$ are positive constants and $\gamma,\rho>1$, and $X_{t-}:=\lim_{s\rightarrow t-}X_s$. In what follows, the jump coefficient $h:\mathbb{R}\!\rightarrow\! \mathbb{R}$ is assumed to be deterministic for simplicity. Here the scalar Wiener process $W_{t}$ and Poisson process $N_{t}$ with intensity $\lambda\!>\!0$ are both defined on a complete probability space $(\Omega, \mathcal{F}, \mathbb{P})$ with a normal filtration $\{\mathcal{F}_t\}_{t\geq 0}$. Note that the two processes mentioned above are independent of each other and the compensated Poisson process $\tilde{N}_{t}:=N_{t}-\lambda t $ is a martingale, which is a key component in our analysis. We assume that the given initial value $x_0 \in\mathbb{R}_+$.

In view of the fact that the exact solution to \eqref{1.1} is not available and development and analysis of numerical method for simulation of the problem are of significant interest in practice, as the objective of this paper, we plan to present an efficient numerical method for \eqref{1.1}. Firstly, we note that the exact solution of \eqref{1.1} takes values in positive domain $(0,\infty)$ under the conditions imposed on jump coefficient $h$ and parameters $\gamma,\rho$, see Proposition \ref{pro2.2}. Therefore, it is necessary to construct a positivity-preserving numerical method for \eqref{1.1}. Secondly, when $h\equiv 0$, the model \eqref{1.1} reduces to the well-known Ait--Sahalia-type interest rate model. For this reduced problem, various time discretization methods have been designed and analyzed. For example, as already shown in \cite{Hutzenthaler2011Strong}, the classical Euler method produces divergent numerical approximation. To this end, Lukasz Szpruch et al. in \cite{Szpruch2011Numerical} proposed the backward Euler method (BEM) to approximate the solution of this reduced problem and found that the numerical method converges strongly to the true solution without revealing any convergence rate. A few years later, Andreas Neuenkirch et al. in \cite{Neuenkirch2014First} introduced the Lamperti transformation. Under appropriate assumptions, the transformed SDE was discretized by BEM and was transformed back so that an approximate solution is inside the domain of the original solution, where the $p$-th moment convergence rate of their scheme was proved to be one. If we want to generalize the above optimal result of this reduced model to the jump-extended model, we finally need to employ the idea of the above numerical method to obtain first-order convergence rate for the jump-extended model with respect to $L^{p}$-error criterion in finite time intervals.

For jump-extended model, as described in \cite{Platen2010Numerical}, the discrete-time approximations considered are divided into regular and jump-adapted methods. Regular methods employ time discretizations that do not include the jump times of the Poisson jumps. Jump-adapted time discretizations, on the other hand, include these jump times. At present, for the jump-extended SDE $\eqref{1.1}$, there are several numerical methods using regular time discretizations, like the Euler method in \cite{Deng2019Generalized}, where they presented the analytical properties including positivity of the exact solution, and proved that the numerical solution converges to the exact solution of the model only in probability. Besides, it is proved in \cite{Zhao2021On} that the BEM is positivity-preserving and strongly convergent with order only one-half in the mean-square sense for full parameters in the case $\gamma+1>2\rho$ and for parameters obeying $\frac{\alpha_2}{\alpha_3^2}>2\gamma-\frac{3}{2}$ in the general critical case $\gamma+1=2\rho$. To obtain the positive numerical solution and recover the $p$-th moment convergence rate of order one for $\eqref{1.1}$, we introduce the so-called transformed jump-adapted backward Euler method, which combining the idea of the Lamperti-backward Euler approximation in \cite{Neuenkirch2014First} with the idea of the jump-adapted methods. More precisely, using the Lamperti transformation, we transform $\eqref{1.1}$ into a jump-extended SDE with additive noise. Then we apply the jump-adapted backward Euler method (JABEM) with a jump-adapted time discretization to the transformed jump-extended SDE. Finally, transforming back yields a numerical approximation for the original model $\eqref{1.1}$, namely TJABEM. 

 It is known that the jump-adapted method is generally used in the jump-extended SDEs whose coefficients meet globally Lipschitz condition (see \cite{Platen2010Numerical,Bruti-Liberati2007Strong,Maghsoodi1996Mean}), but very few works (see \cite{Xu2017Transformed}) in nonlinear SDEs with non-globally Lipschitz condition. Here we note that the non-globally Lipschitz continuous drift and diffusion coefficients of the model \eqref{1.1} have brought many difficulties to the analysis of the $p$-th moment convergence rate of the TJABEM. In addition, we need to overcome the following difficulties:
\begin{itemize}
\item the adapted time discretization including all jump times is path-dependent.
\item the error propagation in the inverse transformation needs to be controlled by the boundedness of inverse moments of the JABEM.
\end{itemize}

In this paper, we show that the numerical solution of the TJABEM is inside the domain of the exact solution of $\eqref{1.1}$ under appropriate hypotheses, see Lemma $\ref{lem3.1}$ and $\eqref{4.1}$. Furthermore, for \eqref{1.1}, this numerical method enables us to achieve the expected $p$-th moment convergence rate for the first time, see Theorem \ref{thm4.1}. The remainder of this paper is structured as follows. In the next section, we present properties of the considered problem $\eqref{1.1}$, including the existence and uniqueness of a positive global solution and the boundedness of moments. In Section 3, we propose the transformed jump-adapted backward Euler method and then prove that this method is $p$-th moment convergent with order one. In Section 4, we carry out numerical experiments to support our theoretical results. At last, some proofs are given in Appendix. 

\section{Ait--Sahalia-type model with Poisson jumps}
\subsection{Positive global solution}
We now introduce some notations used in this paper. Let $a\wedge b:=\min\{a,b\}$ and $a\vee b:=\max\{a,b\}$. Let $\mathcal{F}^{W}_{t}:=\sigma(W_{s},0\leq s\leq t)$ denote the natural filtration generated by the Wiener process $W_{t}$ and $\mathcal{F}^{N}_{t}:=\sigma(N_s,0\leq s\leq t)$ denote the natural filtration generated by the Poisson process $N_t$. Define $\mathcal{F}_{t}=\sigma(\mathcal{F}_{s}^{W}\cup\mathcal{F}_{s}^{N},0\leq s\leq t)$, augmented by all $\mathbb{P}$-null sets of $\mathcal{F}$. From now on, we will work on the filtered probability space $(\Omega, \mathcal{F},(\mathcal{F}_t)_{t\geq0},\mathbb{P})$.

The well-definedness of this model $\eqref{1.1}$ in the case of $h=0$ has been given by \cite{Szpruch2011Numerical}. In the case of $h(x)\!=\!\delta x$, where constant $\delta\!>\!0$, the well-definedness of the corresponding model has been given by \cite{Deng2019Generalized}. When $h$ satisfies the following more general condition, Proposition 1 in \cite{Zhao2021On} proved that a unique global solution exists and remains in $\mathbb{R}_+:=(0,\infty)$.

\begin{assumption}\label{ass2.2}
The jump coefficient $h$ is continuously differentiable and there exist constants $\mu, r>0$ such that
\begin{eqnarray}\label{2.1}
	\vert h'(x)\vert \leq \mu \ \ and \ \  x+h(x)\geq rx,\ \ \forall x>0.
\end{eqnarray}
\end{assumption}

\begin{proposition}[\cite{Zhao2021On}]\label{pro2.2}
Let Assumption $\ref{ass2.2}$ hold. Then for any given initial value $X_{0}=x_{0}>0$ and constants $\alpha_{-1}, \alpha_{0},\alpha_{1},\alpha_{2},\alpha_{3}>0$, $\gamma,\rho>1$, the problem $\eqref{1.1}$ admits a unique positive global solution, which almost surely satisfies
\begin{eqnarray}\label{3}
	X_t&=&x_{0}+\int_{0}^{t}(\alpha_{-1}X_{s-}^{-1}-\alpha_0+\alpha_1 X_{s-}-\alpha_2 X_{s-}^{\gamma})\,ds\nonumber\\
	&&+\int_{0}^{t}\alpha_3 X_{s-}^{\rho}\,dW_s+\int_{0}^{t}h(X_{s-})\,dN_s,\ \  t\geq 0.
\end{eqnarray}
\end{proposition}

\subsection{Boundedness of moments}
Throughout this paper, we use $\mathbb{N}$ to denote the set of all positive integers and let $M\in\mathbb{N}$, $T\in(0,\infty)$ be given. Define the conditional expectation $\mathbb{E}^{N}[X]=\mathbb{E}[X\vert \mathcal{F}^{N}_{T}].$  We always assume that $C$ stands for generic positive constants that are independent of the discretization parameters and whose values might change every time as it appears.
\begin{proposition}[\cite{Zhao2021On}]\label{pro2.3}
Let Assumption $\ref{ass2.2}$ hold. If one of the following conditions holds:
\begin{itemize}
	\item $\gamma>2\rho-1$, $p\geq 2$,
	\item $\gamma=2\rho-1$, $2\leq p<\frac{2\alpha_{2}}{\alpha_{3}^{2}}+1$,
	\item $\gamma\geq 2\rho-1$, $p\leq (-1)\wedge(1-\gamma)$,
\end{itemize}
 then solution $X_{t}$ given by \eqref{3} satisfies
\begin{eqnarray}\label{2.3}
	\sup_{t\in[0,T]}\mathbb{E}[\vert X_{t}\vert^{p}]\leq C(1+\vert x_{0}\vert^{p}),
\end{eqnarray}
where $C$ depends on $p$ and $T$. 
\end{proposition}

By Proposition \ref{pro2.3}, we can further obtain the following proposition.
\begin{proposition}\label{pro2.4}
Let Assumption $\ref{ass2.2}$ hold. If one of the following conditions holds:
\begin{itemize}
	\item $\gamma>2\rho-1$, $p\in(-\infty,\infty)$,
	\item $\gamma=2\rho-1$, $p\in(-\infty,\frac{\alpha_{2}}{\alpha_{3}^{2}}-\rho+\frac{3}{2})$,
\end{itemize}
 then solution $X_t$ given by \eqref{3} satisfies
\begin{eqnarray}
	\mathbb{E}\big[\sup_{t\in[0,T]}\vert X_{t}\vert^{p}\big]<\infty.
\end{eqnarray}
\end{proposition}

\begin{proof}
 Define two functions $f(x)=\alpha_{-1}x^{-1}-\alpha_{0}+\alpha_{1}x-\alpha_{2}x^{\gamma}$ and $g(x)=\alpha_{3} x^{\rho}$. For any $C^{2}$-function $V:\mathbb{R}_{+}\rightarrow\mathbb{R}_{+}$, we introduce an operator $\mathcal{L}V$ from $\mathbb{R}_{+}$ to $\mathbb{R}$ by $\mathcal{L}V(x)=\frac{\partial V(x)}{\partial x}f(x)+\frac{1}{2}\frac{\partial^{2} V(x)}{\partial x^{2}}\big(g(x)\big)^{2}$. According to the proof of Lemma 1 in \cite{Zhao2021On}, for the Lyapunov function $V(x)=x^{p}$, under either condition $\gamma> 2\rho-1, p\geq 2$ which implies $p+\gamma-1>p+2\rho-2$ or condition $\gamma=2\rho-1, 2\leq p<\frac{2\alpha_{2}}{\alpha_{3}^{2}}+1$ which implies $p+\gamma-1=p+2\rho-2$, there is a constant $K_{1}>0$ such that for any $x\in\mathbb{R}_{+}$,
\begin{eqnarray}
&&\mathcal{L}V(x)+\lambda\big(V(x+h(x))-V(x)\big)\nonumber\\
&=&\big[p(\alpha_{-1}x^{p-2}-\alpha_{0}x^{p-1}+\alpha_{1} x^{p}-\alpha_{2}x^{p+\gamma-1})+0.5\alpha_{3}^{2}p(p-1)x^{p+2\rho-2}\big]\nonumber\\
&&+\lambda \big[(x+h(x))^{p}-x^{p}\big]\nonumber\\
&\leq& K_{1}.
\end{eqnarray}
Since $\alpha_{2}/\alpha_{3}^{2}-\rho+\frac{3}{2}<2\alpha_{2}/\alpha_{3}^{2}+1$, then under either condition $\gamma> 2\rho-1, p\geq 2$ or condition $\gamma=2\rho-1, 2\leq p<\frac{\alpha_{2}}{\alpha_{3}^{2}}-\rho+\frac{3}{2}$, using It\^{o}'s formula leads to
\begin{eqnarray}\label{2.8}
\mathbb{E}\big[\sup_{t\in[0,T]}\vert X_{t}\vert^{p}\big]&\leq& x_{0}^{p}+K_{1}T+\mathbb{E}\Big[\sup_{t\in[0,T]}\int_{0}^{t}p\alpha_{3} X_{s-}^{p-1+\rho}\,dW_s\Big]\nonumber\\
&&+\mathbb{E}\Big[\sup_{t\in[0,T]}\int_{0}^{t}\Big(\big(X_{s-}+h(X_{s-})\big)^{p}-X_{s-}^{p}\Big)\,d\tilde{N}_s\Big].
\end{eqnarray}
It follows from Assumption $\ref{ass2.2}$ that
\begin{eqnarray}\label{2.4}
\big(x+h(x)\big)^{p}&\leq& 2^{p-1}x^{p}+2^{p-1}\big\vert h(x)-h(0)+h(0)\big\vert^{p}\nonumber\\
&\leq& \big(2^{2(p-1)}\mu^{p}+2^{p-1}\big)x^{p}+2^{2(p-1)}\big\vert h(0)\big\vert^{p}.
\end{eqnarray}
 Then the H\"{o}lder inequality, Burkholder--Davis--Gundy (BDG) inequality (see \cite{Mao2008Stochastic}), Lemma 2.2 in \cite{Deng2019Generalized}, $\eqref{2.4}$ and Proposition \ref{pro2.3} ensure that
\begin{eqnarray}\label{2.9}
&&\mathbb{E}\big[\sup_{t\in[0,T]}\big\vert X_{t}\big\vert^{p}\big]\nonumber\\
&\leq& x_{0}^{p}+K_{1}T+\Big(\mathbb{E}\Big[\sup_{t\in[0,T]}\big\vert\int_{0}^{t}X_{s-}^{p-1+\rho}\,dW_s\big\vert^{2}\Big]\Big)^{1/2}\nonumber\\
&&+\Big(\mathbb{E}\Big[\sup_{t\in[0,T]}\big\vert\int_{0}^{t}\Big(\big(X_{s-}+h(X_{s-})\big)^{p}-X_{s-}^{p}\Big)\,d\tilde{N}_s\big\vert^{2}\Big]\Big)^{1/2}\nonumber\\
&\leq& x_{0}^{p}+K_{1}T+C\Big(\int_{0}^{T}\mathbb{E}\big[\vert X_{s-}\vert^{2(p-1+\rho)}\big]\,ds\Big)^{1/2}\nonumber\\
&&+C\Big(\int_{0}^{T}\mathbb{E}\Big[\big\vert\big(X_{s-}+h(X_{s-})\big)^{p}-X_{s-}^{p}\big\vert^{2}\Big]\,ds\Big)^{1/2}\nonumber\\
&\leq& C+C\Big(\int_{0}^{T}\mathbb{E}\big[\vert X_{s-}\vert^{2(p-1+\rho)}\big]\,ds\Big)^{1/2}+C\Big(\int_{0}^{T}\mathbb{E}\big[\vert X_{s-}\vert^{2p}\big]\,ds\Big)^{1/2}\nonumber\\
&<&\infty.
\end{eqnarray}
Now referring to the proof of Lemma 2 in \cite{Zhao2021On} and estimates $\eqref{2.8}$, $\eqref{2.9}$, we can deduce that for $p\leq (-1)\wedge(1-\gamma)$, $\mathbb{E}\big[\sup_{t\in[0,T]}\vert X_{t}\vert^{p}\big]<\infty$. The following inequalities
\begin{eqnarray}
x^{p}\leq C(1+x^{2}), \ \ \forall x>0, p\in[0,2)\nonumber
\end{eqnarray}
and
\begin{eqnarray}
x^{p}\leq C(1+x^{(-1)\wedge(1-\gamma)-\varepsilon}),\ \ \forall x>0, \varepsilon>0, p\in\big((-1)\wedge(1-\gamma), 0\big)\nonumber
\end{eqnarray}
 respectively imply that the result $\mathbb{E}\big[\sup_{t\in[0,T]}\vert X_{t}\vert^{p}\big]<\infty$ holds for $p\in[0,2)$ and $p\in\big((-1)\wedge(1-\gamma), 0\big)$. Consequently, we obtain the desired assertions. 
\end{proof}

\section{Numerical method and strong convergence rate}
The aim of the present section is to derive the strong convergence rate of a proposed numerical method for $\eqref{1.1}$. Since \cite{Neuenkirch2014First} obtains the first-order $p$-th moment convergence rate for Ait--Sahalia-type model without jump by Lamperti transformation $Z_{t}=X_{t}^{1-\rho}$, we employ this technique to get a transformed jump-extended SDE with additive noise in this section, which will help us to obtain the expected convergence result of the numerical method for the original model $\eqref{1.1}$.

\subsection{Jump-extended SDE with additive noise}
Using the transformed process $Z_{t}=X_{t}^{1-\rho}$ and It\^{o}'s formula leads to the transformed jump-extended SDE
\begin{eqnarray}\label{3.1}
	dZ_{t}\!&=&\!F_{\gamma,\rho}(Z_{t-})dt+(1-\rho)\alpha_{3}\,dW_t\nonumber\\
\!&&\!+\big[\big(Z_{t-}^{\frac{1}{1-\rho}}+h(Z_{t-}^{\frac{1}{1-\rho}})\big)^{1-\rho}-Z_{t-}\big]dN_t, \ \ t\in(0,T], \ \ Z_{0}=x_{0}^{1-\rho},
\end{eqnarray}
where
\begin{eqnarray}\label{11}
	F_{\gamma,\rho}(x):=(\rho-1)(-\alpha_{-1}x^{\frac{\rho+1}{\rho-1}}+\alpha_{0}x^{\frac{\rho}{\rho-1}}-\alpha_{1}x+\alpha_{2}x^{-\frac{\gamma-\rho}{\rho-1}}+\frac{\rho\alpha_{3}^{2}}{2}x^{-1}).
\end{eqnarray}
 For the above transformed jump-extended model, we will propose a numerical method, which is positivity-preserving and $p$-th moment convergent with order one. Futhermore, we will show that the inverse moments of the numerical solution are bounded. Based on these conclusions, we can obtain the proof of the main result later.

\subsection{The jump-adapted backward Euler method}
To reduce the complexity of higher order method, for the transformed problem $\eqref{3.1}$, we consider the jump-adapted approximation. For any given step size $\Delta t=T/M$, we introduce a deterministic partition
$$
\mathcal{T}^0=\{0=t_{0}^0<t_{1}^0<...<t_M^0=T\}
$$
of the interval $[0, T]$, where $t_{i}^0=i \Delta t, i=0, 1, \cdots, M$. Meanwhile, there may be a random partition $\mathcal{T}^1=\{0\leq \nu_1<\nu_2, \cdots\leq T\}$ of interval $[0, T]$ generated by the Poisson jumps, which depends on sample path. For each sample path, we merge partition $\mathcal{T}^0$ and $\mathcal{T}^1$ to form a new partition
$$
\mathcal{T}=\{0=t_{0}<t_{1}<...<t_{n_{T}}=T\},
$$
where $n_{T}$ is the subscript corresponding to the last time node $T$. The following graph from \cite{Platen2010Numerical} shows how a jump-adapted time discretization  $\mathcal{T}$ is formed. 
\begin{figure}[!ht]\centering
\includegraphics[height=3cm,width=7.4cm]{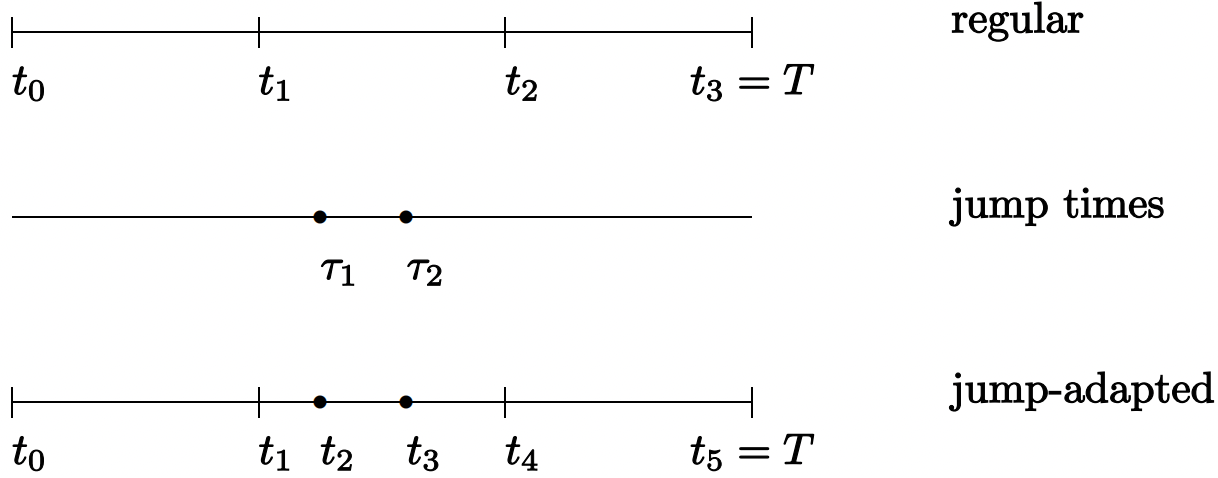}
\end{figure}
This means that $\mathcal{T}$ is path-dependent and the maximum time step size of the resulting jump-adapted discretization is not more than $\Delta t$ in this way.\par
On the mesh $\mathcal{T}$, for any $k\in\{0,1,...,n_{T}-1\}$, the exact solution of $\eqref{3.1}$ can be rewritten as
\begin{equation}\label{3.3}
	\left\{
	\begin{array}{rl}
	Z_{t_{k+1}-}=&Z_{t_{k}}+\int_{t_{k}}^{t_{k+1}}F_{\gamma,\rho}(Z_{t-})\,dt+(1-\rho)\alpha_{3}\Delta W_{k},\\
	Z_{t_{k+1}}=&Z_{t_{k+1}-}+\big[\big(Z_{t_{k+1}-}^{\frac{1}{1-\rho}}+h(Z_{t_{k+1}-}^{\frac{1}{1-\rho}})\big)^{1-\rho}-Z_{t_{k+1}-}\big]\Delta N_{k},
	\end{array}
\right.
\end{equation}
where we have $\Delta N_{k}=1$ if $t_{k+1}$ is a jump time and $\Delta N_{k}=0$ otherwise. The JABEM for $\eqref{3.1}$ is defined by $\bar{Z}_{0}=Z_{0}$ and for $k\in\{0,1,...,n_{T}-1\}$,
\begin{equation}\label{3.4}
	\left\{
	\begin{array}{rl}
	\bar{Z}_{t_{k+1}-}=&\bar{Z}_{t_{k}}+F_{\gamma,\rho}(\bar{Z}_{t_{k+1}-})\Delta t_{k}+(1-\rho)\alpha_{3}\Delta W_{k},\\
	\bar{Z}_{t_{k+1}}=&\bar{Z}_{t_{k+1}-}+\big[\big(\bar{Z}_{t_{k+1}-}^{\frac{1}{1-\rho}}+h(\bar{Z}_{t_{k+1}-}^{\frac{1}{1-\rho}})\big)^{1-\rho}-\bar{Z}_{t_{k+1}-}\big]\Delta N_{k},
	\end{array}
\right.
\end{equation}
where $\Delta t_{k}=t_{k+1}-t_{k}$, $\Delta W_{k}=W_{t_{k+1}}-W_{t_{k}}$ and $\Delta N_{k}=N_{t_{k+1}}-N_{t_{k}}$. A further closer look at $\eqref{3.4}$ suggests that  if $t_{k+1}$ is a jump time, we have
$$\bar{Z}_{t_{k+1}}=\big(\bar{Z}_{t_{k+1}-}^{\frac{1}{1-\rho}}+h(\bar{Z}_{t_{k+1}-}^{\frac{1}{1-\rho}})\big)^{1-\rho},$$
and $\bar{Z}_{t_{k+1}}=\bar{Z}_{t_{k+1}-}$, otherwise. Note that the first derivative of $F_{\gamma,\rho}$ is given by
\begin{eqnarray}\label{a14}
	F'_{\gamma,\rho}(x)&=&-\alpha_{-1}(\rho+1)x^{\frac{2}{\rho-1}}+\alpha_{0}\rho x^{\frac{1}{\rho-1}}-\alpha_{1}(\rho-1)\nonumber\\
&&-\alpha_{2}(\gamma-\rho)x^{-\frac{\gamma-1}{\rho-1}}-0.5(\rho-1)\rho\alpha_{3}^{2}x^{-2}.
\end{eqnarray}
$\gamma\geq 2\rho-1$ implies that $\frac{\gamma-1}{\rho-1}\geq 2$. Hence one can infer that $\lim_{x\rightarrow 0}F'_{\gamma,\rho}(x)=\lim_{x\rightarrow \infty}F'_{\gamma,\rho}(x)=-\infty$. As a result, there exists a constant $Q_{\gamma,\rho}\in[0,\infty)$ such that $\sup_{x\in\mathbb{R}_{+}}F'_{\gamma,\rho}(x)\leq Q_{\gamma,\rho}$. Obviously, function $F'_{\gamma,\rho}$ is continuous in positive domain and satisfies
\begin{eqnarray}\label{3.6}
	(x-y)(F_{\gamma,\rho}(x)-F_{\gamma,\rho}(y))\leq Q_{\gamma,\rho}(x-y)^{2}, \ \ \forall x,y\in\mathbb{R}_{+}.
\end{eqnarray}
Therefore, with the aid of Lemma 2.3 in \cite{Neuenkirch2014First}, we claim that the numerical method $\eqref{3.4}$ is well-defined and positivity-preserving.

\begin{lemma}\label{lem3.1}
Let Assumption \ref{ass2.2} hold and let $\gamma\geq 2\rho-1$. If $Q_{\gamma,\rho}\Delta t<1$, then the numerical method $\eqref{3.4}$ admits a unique solution and $\mathbb{P}\big(\{\bar{Z}_{t_{k+1}}>0\vert\bar{Z}_{t_{k}}>0\}\big)=1$.
\end{lemma}

\begin{proof}
We have $\bar{Z}_{0}=x_{0}^{1-\rho}>0$. By $Q_{\gamma,\rho}\Delta t< 1$, it follows from Lemma 2.3 in \cite{Neuenkirch2014First} that the BEM in the above setting is positivity preserving. Thus, in view of $\eqref{3.4}$, $\bar{Z}_{t_{k}}>0$ implies $\bar{Z}_{t_{k+1}-}>0$, with probability one. By Assumption $\ref{ass2.2}$, we have $x+h(x)>0$ for all $x>0$. If $\Delta N_{k}=1$, then
$$\bar{Z}_{t_{k+1}}=\big(\bar{Z}_{t_{k+1}-}^{\frac{1}{1-\rho}}+h(\bar{Z}_{t_{k+1}-}^{\frac{1}{1-\rho}})\big)^{1-\rho}>0,\ \ a.s.$$
If $\Delta N_{k}=0$, then $\bar{Z}_{t_{k+1}}=\bar{Z}_{t_{k+1}-}>0$, a.s. Consequently, the numerical method $\eqref{3.4}$ is well-defined and positivity-preserving.
\end{proof}

\subsection{Convergence rate for the transformed jump-extended SDE}
In this part, we formulate the convergence rate of the JABEM for the jump-extended SDE $\eqref{3.1}$. According to Proposition $\ref{pro2.4}$, we have the following result.

\begin{corollary}\label{cor3.2}
Let Assumption $\ref{ass2.2}$ hold. If one of the following conditions holds:
\begin{itemize}
	\item $\gamma>2\rho-1$, $p\in(-\infty,\infty)$,
	\item $\gamma=2\rho-1$, $p\in(-\frac{\alpha_{2}/\alpha_{3}^{2}-\rho+\frac{3}{2}}{\rho-1},\infty)$,
\end{itemize}
 then the exact solution of the jump-extended SDE \eqref{3.1} satisfies
\begin{eqnarray}
\mathbb{E}\big[\sup_{t\in[0,T]}\vert Z_{t}\vert^{p}\big]<\infty.
\end{eqnarray}
\end{corollary}

Between jump times the evolution of SDE $\eqref{3.1}$ is that of a diffusion without jumps. By Lemmas 2.1, 2.2 in \cite{Szpruch2011Numerical} and referring to the proof of Proposition $\ref{pro2.4}$, under the same conditions of Corollary $\ref{cor3.2}$, we have 
\begin{eqnarray}\label{b17}
	\mathbb{E}\big[\sup_{t\in[0,T]}\vert Z_{t-}\vert^{p}\big]<\infty.
\end{eqnarray}
Recalling functions $F_{\gamma,\rho}$ and $F'_{\gamma,\rho}$, we show
\begin{eqnarray}
	\!\!F''_{\gamma,\rho}(x)\!&=&\!-\frac{2\alpha_{-1}(\rho+1)}{\rho-1}x^{\frac{3-\rho}{\rho-1}}\!+\!\frac{\alpha_{0}\rho}{\rho-1}x^{\frac{2-\rho}{\rho-1}}\!+\!\frac{\alpha_{2}(\gamma-\rho)(\gamma-1)}{\rho-1}x^{-\frac{\gamma+\rho-2}{\rho-1}}\nonumber\\
	&&+(\rho-1)\rho\alpha_{3}^{2}x^{-3}.
\end{eqnarray}
If $\gamma>2\rho-1$, for any $q\geq 1$, we calculate
\begin{eqnarray}\label{3.9}
	\!\!\!\!\!&&\!\!\!\!\!\mathbb{E}\Big[\sup_{t\in[0,T]}\big\vert(F'_{\gamma,\rho}F_{\gamma,\rho})(Z_{t-})\!+\!\frac{1}{2}(1\!-\!\rho)^{2}\alpha_{3}^{2}F''_{\gamma,\rho}(Z_{t-})\big\vert^{2q}\Big]\!+\!\mathbb{E}\Big[\sup_{t\in[0,T]}\big\vert F'_{\gamma,\rho}(Z_{t-})\big\vert^{2q}\Big]\nonumber\\
	\!\!&<&\!\!\infty.
\end{eqnarray}
If $\gamma=2\rho-1$, for any $1\leq q<\frac{\alpha_{2}/\alpha_{3}^{2}-\rho+\frac{3}{2}}{6(\rho-1)}$, we have $(-3)\times 2q>-\frac{\alpha_{2}/\alpha_{3}^{2}-\rho+\frac{3}{2}}{\rho-1}$ and it is easy to verify that \eqref{3.9} holds. 

To obtain the convergence rate, we need the following additional condition imposed on jump coefficient $h$.
\begin{assumption}\label{ass3.3}
	There are two constants $\mu_{1}(\rho),\mu_{2}(\rho)$, satisfying $0<\mu_{1}(\rho)\leq \mu_{2}(\rho)<\infty$, such that
	\begin{eqnarray}
	\big(1+h(x)/x\big)^{-\rho}\big(1+h'(x)\big)\in[\mu_{1},\mu_{2}], \ \ \forall x>0.	
	\end{eqnarray}
\end{assumption}

\begin{remark}[\cite{Xu2017Transformed}]
We highlight that the family of jump coefficient $h$ satisfying Assumptions $\ref{ass2.2}$ and $\ref{ass3.3}$ is more general than the linear functions studied in \cite{Deng2019Generalized}. Evidently, the assumptions can be fulfilled if there exist constants $r_{1}, r_{2}, L_{2}>0$ and $L_{1}>-1$ such that $r_{1}x\leq x+h(x) \leq r_{2}x$ and $L_{1}\leq h'(x)\leq L_{2}$ for all $x>0$, which allows for the linear function $h(x)=\varrho x$ and also some nonlinear functions like $h(x) =\varrho\sin{x}$, $h(x)=\frac{\varrho x}{1+x}$ for $\varrho>-1$ and so on.
\end{remark}

\begin{theorem}\label{thm3.5}
Let Assumptions $\ref{ass2.2}$ and $\ref{ass3.3}$ hold. If one of the following conditions holds:
\begin{itemize}
	\item $\gamma>2\rho-1$, $\eta\geq 1$,
	\item $\gamma=2\rho-1$, $\eta\in[1, \frac{\alpha_{2}/\alpha_{3}^{2}-\rho+\frac{3}{2}}{12(\rho-1)})$,
\end{itemize}
 then the exact solution of $\eqref{3.1}$ and numerical solution given by \eqref{3.4} satisfy
\begin{eqnarray}\label{3.11}
	\mathbb{E}\big[\sup_{k=0,1,...,n_{T}}\big\vert\bar{Z}_{t_{k}}-Z_{t_{k}}\big\vert^{\eta}\big]\leq C(\Delta t)^{\eta}.
\end{eqnarray}	
\end{theorem}

\begin{proof}
We do not measure the approximation error $\eqref{3.11}$ directly. Instead, we turn to the discrepancy between the intermediate solutions, that is, $\mathbb{E}\big[\sup_{k=0,1,...,n_{T}}\vert\bar{Z}_{t_{k}-}-Z_{t_{k}-}\vert^{\eta}\big]$ for any $\eta\geq 1$. According to $\eqref{3.3}$ and $\eqref{3.4}$, for $k\in\{1,2,...,n_{T}-1\}$, we have
\begin{eqnarray}
	Z_{t_{k+1}-}&=&Z_{t_{k}-}+\big[\big(Z_{t_{k}-}^{\frac{1}{1-\rho}}+h(Z_{t_{k}-}^{\frac{1}{1-\rho}})\big)^{1-\rho}-Z_{t_{k}-}\big]\Delta N_{k-1}\nonumber\\
	&&+\int_{t_{k}}^{t_{k+1}}F_{\gamma,\rho}(Z_{t-})\,dt+(1-\rho)\alpha_{3}\Delta W_{k},\label{3.12}\\
	\bar{Z}_{t_{k+1}-}&=&\bar{Z}_{t_{k}-}+\big[\big(\bar{Z}_{t_{k}-}^{\frac{1}{1-\rho}}+h(\bar{Z}_{t_{k}-}^{\frac{1}{1-\rho}})\big)^{1-\rho}-\bar{Z}_{t_{k}-}\big]\Delta N_{k-1}\nonumber\\
	&&+F_{\gamma,\rho}(\bar{Z}_{t_{k+1}-})\Delta t_{k}+(1-\rho)\alpha_{3}\Delta W_{k}.\label{3.13}
\end{eqnarray}
Note that the process $Z_{t-}$ is a solution of SDE without jump when $t\in[t_{k},t_{k+1})$ and we have
\begin{eqnarray}
	Z_{t_{k+1}-}=Z_{t-}+\int_{t}^{t_{k+1}}F_{\gamma,\rho}(Z_{s-})\,ds+\int_{t}^{t_{k+1}}\alpha_{3}(1-\rho)\,dW_s.
\end{eqnarray}
Consequently, $\eqref{3.12}$ can be rewritten as
\begin{eqnarray}
Z_{t_{k+1}-}&=&Z_{t_{k}-}+\big[\big(Z_{t_{k}-}^{\frac{1}{1-\rho}}+h(Z_{t_{k}-}^{\frac{1}{1-\rho}})\big)^{1-\rho}-Z_{t_{k}-}\big]\Delta N_{k-1}\nonumber\\
	&&+F_{\gamma,\rho}(Z_{t_{k+1}-})\Delta t_{k}+(1-\rho)\alpha_{3}\Delta W_{k}-\Xi_{k+1},\nonumber
\end{eqnarray}
where $\Xi_{k+1}:=\int_{t_{k}}^{t_{k+1}}[F_{\gamma,\rho}(Z_{t_{k+1}-})-F_{\gamma,\rho}(Z_{t-})]\,dt$.
Using It\^{o}'s formula, we deduce
\begin{eqnarray}\label{3.15}
	\Xi_{k+1}&=&\int_{t_{k}}^{t_{k+1}}\int_{t}^{t_{k+1}}\big[(F'_{\gamma,\rho}F_{\gamma,\rho})(Z_{s-})+\frac{1}{2}(1-\rho)^{2}\alpha_{3}^{2}F''_{\gamma,\rho}(Z_{s-})\big]\,ds\,dt\nonumber\\
	&&+(1-\rho)\alpha_{3}\int_{t_{k}}^{t_{k+1}}\int_{t}^{t_{k+1}}F'_{\gamma,\rho}(Z_{s-})\,dW_s\,dt\nonumber\\
	&=&\int_{t_{k}}^{t_{k+1}}(s-t_{k})\big[(F'_{\gamma,\rho}F_{\gamma,\rho})(Z_{s-})+\frac{1}{2}(1-\rho)^{2}\alpha_{3}^{2}F''_{\gamma,\rho}(Z_{s-})\big]\,ds\nonumber\\
	&&+(1-\rho)\alpha_{3}\int_{t_{k}}^{t_{k+1}}(s-t_{k})F'_{\gamma,\rho}(Z_{s-})\,dW_s,
\end{eqnarray}
where we have used stochastic Fubini theorem. It follows from \eqref{3.12} and \eqref{3.13} that
\begin{eqnarray}\label{3.17}
&&\bar{Z}_{t_{k+1}-}-Z_{t_{k+1}-}\nonumber\\
	&=&\bar{Z}_{t_{k}-}-Z_{t_{k}-}+\Xi_{k+1}+\big(F_{\gamma,\rho}(\bar{Z}_{t_{k+1}-})-F_{\gamma,\rho}(Z_{t_{k+1}-})\big)\Delta t_{k}\\
&&+\big[\big(\bar{Z}_{t_{k}-}^{\frac{1}{1-\rho}}+h(\bar{Z}_{t_{k}-}^{\frac{1}{1-\rho}})\big)^{1-\rho}-\big(Z_{t_{k}-}^{\frac{1}{1-\rho}}+h(Z_{t_{k}-}^{\frac{1}{1-\rho}})\big)^{1-\rho}-\bar{Z}_{t_{k}-}+Z_{t_{k}-}\big]\Delta N_{k-1}.\nonumber
\end{eqnarray}
To deal with the last two terms on the right-hand side of the equality above, we set
$$I_{\gamma,\rho,t_{k+1}}:=\int_{0}^{1}F'_{\gamma,\rho}(Z_{t_{k+1}-}+\theta(\bar{Z}_{t_{k+1}-}-Z_{t_{k+1}-}))\,d\theta.$$
Then $F_{\gamma,\rho}(\bar{Z}_{t_{k+1}-})-F_{\gamma,\rho}(Z_{t_{k+1}-})=I_{\gamma,\rho,t_{k+1}}(\bar{Z}_{t_{k+1}-}-Z_{t_{k+1}-})$. Recalling $\eqref{3.6}$ gives
\begin{eqnarray}\label{a3.17}
 I_{\gamma,\rho,t_{k+1}}\leq Q_{\gamma,\rho}, \ \ a.s. 	
\end{eqnarray}
 At the same time, using the mean value theorem yields
\begin{eqnarray}
&&\big(\bar{Z}_{t_{k}-}^{\frac{1}{1-\rho}}+h(\bar{Z}_{t_{k}-}^{\frac{1}{1-\rho}})\big)^{1-\rho}-\big(Z_{t_{k}-}^{\frac{1}{1-\rho}}+h(Z_{t_{k}-}^{\frac{1}{1-\rho}})\big)^{1-\rho}\nonumber\\
&=&\big(1+h(\varsigma_{k-}^{\frac{1}{1-\rho}})/\varsigma_{k-}^{\frac{1}{1-\rho}}\big)^{-\rho}\big(1+h'(\varsigma_{k-}^{\frac{1}{1-\rho}})\big)(\bar{Z}_{t_{k}-}-Z_{t_{k}-}),
\end{eqnarray}
where $\varsigma_{k-}$ is $\mathcal{F}_{t_{k}}$-measurable and  $\varsigma_{k-}$ takes value between $\bar{Z}_{t_{k}-}$ and $Z_{t_{k}-}$. For $k\in\{1,2,...,n_{T}\}$, we set
\begin{eqnarray}
&& e_{k-}:=\bar{Z}_{t_{k}-}-Z_{t_{k}-},\ \ \psi_{k}:=1-I_{\gamma,\rho,t_{k}}\Delta t_{k-1},\label{3.20}\\
&& \chi_{k}:=1+\big(1+h(\varsigma_{k-}^{\frac{1}{1-\rho}})/\varsigma_{k-}^{\frac{1}{1-\rho}}\big)^{-\rho}\big(1+h'(\varsigma_{k-}^{\frac{1}{1-\rho}})\big)\Delta N_{k-1}-\Delta N_{k-1}.\label{3.21}
\end{eqnarray}
Let $\Delta t$ be sufficiently small such that $Q_{\gamma,\rho}\Delta t< \xi$ for $\xi\in(0,1)$. Since $\psi_{k}\geq 1-Q_{\gamma,\rho}\Delta t>0$ a.s., $\eqref{3.17}$ can be abbreviated as
\begin{eqnarray}\label{3.22}
	e_{(k+1)-}=\psi_{k+1}^{-1}\chi_{k}e_{k-}+\psi_{k+1}^{-1}\Xi_{k+1},\ \ k\in\{1,2,...,n_{T}-1\}.
\end{eqnarray}
Repeating the iteration $\eqref{3.22}$, we obtain
\begin{eqnarray}\label{3.23}
	e_{(k+1)-}&=&(\prod_{i=2}^{k+1}\psi_{i}^{-1})(\prod_{j=1}^{k}\chi_{j})e_{1-}+\sum_{i=1}^{k}(\prod_{l=i+1}^{k+1}\psi_{l}^{-1})(\prod_{j=i+1}^{k}\chi_{j})\Xi_{i+1}\nonumber\\
	&=&\sum_{i=0}^{k}(\prod_{l=i+1}^{k+1}\psi_{l}^{-1})(\prod_{j=i+1}^{k}\chi_{j})\Xi_{i+1}\nonumber\\
	&=&(\prod_{l=1}^{k+1}\psi_{l}^{-1})(\prod_{j=1}^{k}\chi_{j})\cdot\sum_{i=0}^{k}(\prod_{l=1}^{i}\psi_{l})(\prod_{j=1}^{i}\chi_{j}^{-1})\Xi_{i+1},
\end{eqnarray}
where $\Xi_{1}:=\psi_{1}e_{1-}$. For $k\in\{1,2,...,n_{T}\}$, we set
$$\Psi_{0}:=1,\ \ \Psi_{k}:=\prod_{l=1}^{k}\psi_{l},\ \ \tilde{\Psi}_{0}:=1,\ \ \tilde{\Psi}_{k}:=\frac{\Psi_{k}}{(1-Q_{\gamma,\rho}\Delta t)^{k}}, \ \ \Gamma_{k}:=\prod_{j=1}^{k}\chi_{j}.$$
It is obvious that $\Psi_{k}, \tilde{\Psi}_{k}$ and $\Gamma_{k}$ are $\mathcal{F}_{t_{k}}$-measurable. Moreover, it is easy to see that $\Psi_{k}>0, \tilde{\Psi}_{k}\geq 1$ and $\tilde{\Psi}_{k}$ is nondecreasing in $k$ almost surely since $I_{\gamma,\rho,t_{k}}\Delta t_{k-1}\leq Q_{\gamma,\rho}\Delta t$ a.s. for all $k\in\{1,2,...,n_{T}\}$ by $\eqref{a3.17}$. For any $\beta\in\mathbb{N}$, we will show that
$$\mathbb{E}\Big[\sup_{k\in\{0,1,...,n_{T}\}}\Big\vert\frac{\tilde{\Psi}_{k}}{\Psi_{k}}\Big\vert^{\beta}\Big]<\infty,  \ \ \mathbb{E}\Big[\sup_{0\leq l\leq k\leq n_{T}}\Big\vert\frac{\Psi_{l}}{\Psi_{k}}\Big\vert^{\beta}\Big]<\infty.$$
Indeed,
\begin{eqnarray}
\mathbb{E}\Big[\sup_{k\in\{0,1,...,n_{T}\}}\Big\vert\frac{\tilde{\Psi}_{k}}{\Psi_{k}}\Big\vert^{\beta}\Big]\!&=&\!\mathbb{E}\big[(1-Q_{\gamma,\rho}\Delta t)^{-\beta n_{T}}\big]\nonumber\\
\!&\leq&\!(1-Q_{\gamma,\rho}\Delta t)^{-\beta M}\mathbb{E}\big[(1-Q_{\gamma,\rho}\Delta t)^{-\beta N(T)}\big]\nonumber\\
\!&\leq&\!(1-Q_{\gamma,\rho}\Delta t)^{-\beta M}\sum_{i=0}^{\infty}(1-\xi)^{-\beta i}e^{-\lambda T}\frac{(\lambda T)^{i}}{i!}\!<\!\infty.
\end{eqnarray}
Moreover, for $k\in\{1,2,...,n_{T}\}$, $\chi_{k}=1$ if $t_{k}$ isn't a jump time and
$$\chi_{k}=\big(1+h(\varsigma_{k-}^{\frac{1}{1-\rho}})/\varsigma_{k-}^{\frac{1}{1-\rho}}\big)^{-\rho}\big(1+h'(\varsigma_{k-}^{\frac{1}{1-\rho}})\big),\ \ \mbox{otherwise}.$$
Assumption $\ref{ass3.3}$ implies $0<(\mu_{1}\wedge1)\leq \chi_{k}\leq (\mu_{2}\vee1)<\infty$. As a result, we have
\begin{eqnarray}
\mathbb{E}\Big[\!\sup_{k\in\{1,2,...,n_{T}\}}\!\big\vert\Gamma_{k}\big\vert^{\beta}\Big]\!\!&\leq&\!\! \mathbb{E}\big[(\mu_{2}\vee1)^{\beta N(T)}\big]=\sum_{i=0}^{\infty}(\mu_{2}\vee1)^{\beta i}e^{-\lambda T}\frac{(\lambda T)^{i}}{i!}\!<\!\infty,\nonumber\\
\mathbb{E}\Big[\!\sup_{k\in\{1,2,...,n_{T}\}}\!\!\big\vert\Gamma_{k}\big\vert^{-\beta}\Big]\!\!&\leq&\!\! \mathbb{E}\big[(\mu_{1}\wedge1)^{-\beta N(T)}\big]=\sum_{i=0}^{\infty}(\mu_{1}\wedge1)^{-\beta i}e^{-\lambda T}\frac{(\lambda T)^{i}}{i!}\!<\!\infty.\nonumber
\end{eqnarray}
Let $\lfloor s\rfloor\!:=\!\min\{n\in\{0,1,...,n_{T}\}\colon \sum_{i=0}^{n}\Delta t_{i}\!>\!s\ \ \mbox{or}\ \ \sum_{i=0}^{n-1}\Delta t_{i}\!=\!s\}$ for $s\in[0,T]$, where $\Delta t_{n_{T}}:=0$, and let
$$I_{t}:=\int_{0}^{t}(1-Q_{\gamma,\rho}\Delta t)^{\lfloor s\rfloor}\Gamma_{\lfloor s\rfloor}^{-1}\big(s-t_{\lfloor s\rfloor}\big)F'_{\gamma,\rho}(Z_{s-})\,dW_s.$$
Using the results above, we obtain
\begin{eqnarray}
e_{(k+1)-}&=&\Psi_{k+1}^{-1}\Gamma_{k}\cdot\sum_{i=0}^{k}\Psi_{i}\Gamma_{i}^{-1}\Xi_{i+1}\nonumber\\
&=&\Gamma_{k}\cdot\sum_{i=0}^{k}\int_{t_{i}}^{t_{i+1}}\Psi_{k+1}^{-1}\Psi_{i}\Gamma_{i}^{-1}(s-t_{i})\Big[(F'_{\gamma,\rho}F_{\gamma,\rho})(Z_{s-})\nonumber\\
	&&+\frac{1}{2}(1-\rho)^{2}\alpha_{3}^{2}F''_{\gamma,\rho}(Z_{s-})\Big]\,ds+\Psi_{k+1}^{-1}\Gamma_{k}\cdot(1-\rho)\alpha_{3}\sum_{i=0}^{k}\tilde{\Psi}_{i}(I_{t_{i+1}}-I_{t_{i}})\nonumber\\
&=&\Gamma_{k}\cdot\int_{0}^{t_{k+1}}\Psi_{k+1}^{-1}\Psi_{\lfloor s\rfloor}\Gamma_{\lfloor s\rfloor}^{-1}(s-t_{\lfloor s\rfloor})\nonumber\\
	&&\cdot\Big[(F'_{\gamma,\rho}F_{\gamma,\rho})(Z_{s-})+\frac{1}{2}(1-\rho)^{2}\alpha_{3}^{2}F''_{\gamma,\rho}(Z_{s-})\Big]\,ds\nonumber\\
	&&+\Psi_{k+1}^{-1}\Gamma_{k}\cdot(1-\rho)\alpha_{3}\Big[\tilde{\Psi}_{k}I_{t_{k+1}}+\sum_{i=1}^{k}(\tilde{\Psi}_{i-1}-\tilde{\Psi}_{i})I_{t_{i}}\Big].
\end{eqnarray}
Noting that $\big\vert\tilde{\Psi}_{k}I_{t_{k+1}}+\sum_{i=1}^{k}(\tilde{\Psi}_{i-1}-\tilde{\Psi}_{i})I_{t_{i}}\big\vert\leq \tilde{\Psi}_{k}\vert I_{t_{k+1}}\vert+\sum_{i=1}^{k}(\tilde{\Psi}_{i}-\tilde{\Psi}_{i-1})\vert I_{t_{i}}\vert\leq 2\tilde{\Psi}_{k+1}\sup_{1\leq l\leq k+1}\vert I_{t_{l}}\vert$, and using the H\"{o}lder inequality, BDG inequality and $\eqref{3.9}$, we have
\begin{eqnarray}
	&&\mathbb{E}\Big[\sup_{1\leq k\leq n_{T}}\vert e_{k-}\vert^{\eta}\Big]\nonumber\\
	&\leq& C\Big(\mathbb{E}\Big[\sup_{1\leq k\leq n_{T}}\Big\vert\int_{0}^{t_{k}}\frac{\Psi_{\lfloor s\rfloor}}{\Psi_{k}}\Gamma_{\lfloor s\rfloor}^{-1}(s-t_{\lfloor s\rfloor})\Big[(F'_{\gamma,\rho}F_{\gamma,\rho})(Z_{s-})\nonumber\\
	&&+\frac{1}{2}(1-\rho)^{2}\alpha_{3}^{2}F''_{\gamma,\rho}(Z_{s-})\Big]\,ds\Big\vert^{2\eta}\Big]\Big)^{1/2}+C\mathbb{E}\Big[\sup_{1\leq k\leq n_{T}}\vert\Gamma_{k}\vert^{\eta}\cdot\sup_{1\leq k\leq n_{T}}\Big\vert\frac{\tilde{\Psi}_{k}}{\Psi_{k}}\Big\vert^{\eta}\nonumber\\
&&\cdot\sup_{1\leq k\leq n_{T}}\Big\vert\int_{0}^{t_{k}}(1-Q_{\gamma,\rho}\Delta t)^{\lfloor s\rfloor}\Gamma_{\lfloor s\rfloor}^{-1}(s-t_{\lfloor s\rfloor})F'_{\gamma,\rho}(Z_{s-})\,dW_s\Big\vert^{\eta}\Big]\nonumber\\
&\leq&C(\Delta t)^{\eta}\Big(\mathbb{E}\Big[\sup_{1\leq l\leq k\leq n_{T}}\Big\vert\frac{\Psi_{l}}{\Psi_{k}}\Big\vert^{2\eta}\cdot\sup_{1\leq k\leq n_{T}}\vert\Gamma_{k}\vert^{-2\eta}\cdot\int_{0}^{T}\Big\vert(F'_{\gamma,\rho}F_{\gamma,\rho})(Z_{s-})\nonumber\\
&&+\frac{1}{2}(1-\rho)^{2}\alpha_{3}^{2}F''_{\gamma,\rho}(Z_{s-})\Big\vert^{2\eta}\,ds\Big]\Big)^{1/2}\nonumber\\
&&+C\Big(\mathbb{E}\Big[\Big(\int_{0}^{T}\Big\vert(1-Q_{\gamma,\rho}\Delta t)^{\lfloor s\rfloor}\Gamma_{\lfloor s\rfloor}^{-1}(s-t_{\lfloor s\rfloor})F'_{\gamma,\rho}(Z_{s-})\Big\vert^{2}\,ds\Big)^{\eta}\Big]\Big)^{1/2}\nonumber\\
&\leq&C(\Delta t)^{\eta}\Big(\mathbb{E}\Big[\int_{0}^{T}\Big\vert(F'_{\gamma,\rho}F_{\gamma,\rho})(Z_{s-})+\frac{1}{2}(1-\rho)^{2}\alpha_{3}^{2}F''_{\gamma,\rho}(Z_{s-})\Big\vert^{4\eta}\,ds\Big]\Big)^{1/4}\nonumber\\
&&+C(\Delta t)^{\eta}\Big(\mathbb{E}\Big[\sup_{1\leq k\leq n_{T}}\vert\Gamma_{k}\vert^{-2\eta}\cdot\int_{0}^{T}\vert F'_{\gamma,\rho}(Z_{s-})\vert^{2\eta}\,ds\Big]\Big)^{1/2}\nonumber\\
&\leq&C(\Delta t)^{\eta}\bigg[\Big(\mathbb{E}\Big[\sup_{t\in[0,T]}\vert F'_{\gamma,\rho}(Z_{t-})\vert^{4\eta}\Big]\Big)^{1/4}+\Big(\mathbb{E}\Big[\sup_{t\in[0,T]}\Big\vert(F'_{\gamma,\rho}F_{\gamma,\rho})(Z_{t-})\nonumber\\
&&+\frac{1}{2}(1-\rho)^{2}\alpha_{3}^{2}F''_{\gamma,\rho}(Z_{t-})\Big\vert^{4\eta}\Big]\Big)^{1/4}\bigg]\nonumber\\
&\leq&C(\Delta t)^{\eta},
\end{eqnarray}
where we have used the fact that $(1-Q_{\gamma,\rho}\Delta t)^{\lfloor s\rfloor}\in(0,1)$ for $s\in[0,T]$. Therefore,
 \begin{eqnarray}\label{3.28}
 \mathbb{E}\Big[\sup_{k=0,1,...,n_{T}}\vert\bar{Z}_{t_{k}-}-Z_{t_{k}-}\vert^{\eta}\Big]\leq C(\Delta t)^{\eta}.
 \end{eqnarray}
According to $\eqref{3.3}$ and $\eqref{3.4}$, we have
\begin{eqnarray}\label{3.29}
\big\vert\bar{Z}_{t_{k}}-Z_{t_{k}}\big\vert&=&\Big\vert\bar{Z}_{t_{k}-}-Z_{t_{k}-}+\big[\big(\bar{Z}_{t_{k}-}^{\frac{1}{1-\rho}}+h(\bar{Z}_{t_{k}-}^{\frac{1}{1-\rho}})\big)^{1-\rho}-\big(Z_{t_{k}-}^{\frac{1}{1-\rho}}+h(Z_{t_{k}-}^{\frac{1}{1-\rho}})\big)^{1-\rho}\nonumber\\
&&-\bar{Z}_{t_{k}-}+Z_{t_{k}-}\big]\Delta N_{k-1}\Big\vert\nonumber\\
&\leq&(\mu_{2}\vee1)\big\vert\bar{Z}_{t_{k}-}-Z_{t_{k}-}\big\vert.
\end{eqnarray}
Finally, $\eqref{3.29}$ together with $\eqref{3.28}$ yields the required assertion.
\end{proof}

\subsection{Boundedness of inverse moments of the JABEM}
\begin{lemma}\label{lem3.6}
Let Assumptions \ref{ass2.2} and \ref{ass3.3} hold. If one of the following conditions holds:
\begin{itemize}
	\item $\gamma>2\rho-1$, $q\geq 1$,
	\item $\gamma=2\rho-1$, $q\in[1, \frac{\alpha_{2}/\alpha_{3}^{2}-\rho+\frac{3}{2}}{24(\rho-1)})$,
\end{itemize}
 then the exact solution of \eqref{3.1} and numerical solution given by \eqref{3.4} satisfy
\begin{eqnarray}
\mathbb{E}\Big[\sup_{k=1,2,...,n_{T}}\!\!\big\vert \Delta t_{k-1}F_{\gamma,\rho}(\bar{Z}_{t_{k}-})\big\vert^{2q}\Big]\!\leq\! C(\Delta t)^{2q}\!+\!C\mathbb{E}\Big[\sup_{k=1,2,...,n_{T}}\!\!\big\vert\Delta t_{k-1} F_{\gamma,\rho}(Z_{t_{k}-})\big\vert^{2q}\Big].\nonumber
\end{eqnarray}
\end{lemma}

The proof of this lemma is given in the Appendix. Now we employ the conclusions above to show the following lemma.

\begin{lemma}\label{lem3.7}
Let $m:=\frac{\gamma-\rho}{\rho-1}$ and let the conditions of Lemma \ref{lem3.6} hold. If one of the following conditions holds:
\begin{itemize}
	\item $\gamma>2\rho-1$, $q\geq 1$,
	\item $\gamma=2\rho-1$, $q\in[1, \frac{\alpha_{2}/\alpha_{3}^{2}-\rho+\frac{3}{2}}{24(\rho+1)})$,
\end{itemize}
 then the exact solution of $\eqref{3.1}$ and numerical solution given by \eqref{3.4} satisfy
\begin{eqnarray}
\mathbb{E}\Big[\sup_{k=1,2,...,n_{T}}\!\big(\Delta t_{k-1}\big)^{2q}\bar{Z}_{t_{k}}^{-2mq}\Big]\!\leq\! C(\Delta t)^{2q}+C\mathbb{E}\big[\sup_{k=1,2,...,n_{T}}\big(\Delta t_{k-1}\big)^{2q}Z_{t_{k}}^{-2mq}\big].\nonumber
\end{eqnarray}	
\end{lemma}

\begin{proof}
For some $c_{1}>0$, in view of \eqref{11}, $F_{\gamma,\rho}(x)$ can be rewritten as
$$F_{\gamma,\rho}(x)=c_{1}x^{-m}+u(x),\ \ \forall x>0,$$
where the function $u$ satisfies $\vert u(x)\vert\leq c_{2}(1+x^{\frac{\rho+1}{\rho-1}})$ for all $x>0$ and some $c_{2}>0$. Thus for any $q\geq 1$, we have
\begin{eqnarray}
\!&&\!\mathbb{E}\big[\sup_{k=1,2,...,n_{T}}\big(\Delta t_{k-1}\big)^{2q}\bar{Z}_{t_{k}-}^{-2mq}\big]\nonumber\\
\!&=&\!\mathbb{E}\Big[\sup_{k=1,2,...,n_{T}}\big(\Delta t_{k-1}\big)^{2q}\big\vert\frac{1}{c_{1}}\big(F_{\gamma,\rho}(\bar{Z}_{t_{k}-})-u(\bar{Z}_{t_{k}-})\big)\big\vert^{2q}\Big]	\nonumber\\
\!&\leq&\!C\mathbb{E}\big[\sup_{k=1,2,...,n_{T}}\big(\Delta t_{k-1}\big)^{2q}\big\vert F_{\gamma,\rho}(\bar{Z}_{t_{k}-})\big\vert^{2q}\big]\!+\!C\mathbb{E}\big[\sup_{k=1,2,...,n_{T}}\big(\Delta t_{k-1}\big)^{2q}\big\vert u(\bar{Z}_{t_{k}-})\big\vert^{2q}\big].\nonumber
\end{eqnarray}
By Lemma $\ref{lem3.6}$, we arrive at
\begin{eqnarray}
\mathbb{E}\Big[\sup_{k=1,2,...,n_{T}}\!\!\big\vert \Delta t_{k-1}F_{\gamma,\rho}(\bar{Z}_{t_{k}-})\big\vert^{2q}\Big]\!\leq\! C(\Delta t)^{2q}\!+\!C\mathbb{E}\Big[\sup_{k=1,2,...,n_{T}}\!\!\big\vert\Delta t_{k-1} F_{\gamma,\rho}(Z_{t_{k}-})\big\vert^{2q}\Big].\nonumber
\end{eqnarray}
It is straightforward to verify that
\begin{eqnarray}
	&&\mathbb{E}\big[\sup_{k=1,2,...,n_{T}}\big\vert \Delta t_{k-1} F_{\gamma,\rho}(Z_{t_{k}-})\big\vert^{2q}\big]\nonumber\\
&=&\mathbb{E}\big[\sup_{k=1,2,...,n_{T}}\big(\Delta t_{k-1}\big)^{2q}\big\vert c_{1}Z_{t_{k}-}^{-m}+u(Z_{t_{k}-})\big\vert^{2q}\big]\nonumber\\
	&\leq&C\mathbb{E}\big[\sup_{k=1,2,...,n_{T}}\big(\Delta t_{k-1}\big)^{2q} Z_{t_{k}-}^{-2mq}\big]+C\mathbb{E}\big[\sup_{k=1,2,...,n_{T}}\big(\Delta t_{k-1}\big)^{2q}\big\vert u(Z_{t_{k}-})\big\vert^{2q}\big].\nonumber
\end{eqnarray}
Consequently,
\begin{eqnarray}\label{a3.61}
&&\mathbb{E}\big[\sup_{k=1,2,...,n_{T}}\big(\Delta t_{k-1}\big)^{2q}\bar{Z}_{t_{k}-}^{-2mq}\big]\nonumber\\
&\leq& C\mathbb{E}\big[\sup_{k=1,2,...,n_{T}}\big(\Delta t_{k-1}\big)^{2q}Z_{t_{k}-}^{-2mq}\big]+C\mathbb{E}\big[\sup_{k=1,2,...,n_{T}}\big(\Delta t_{k-1}\big)^{2q}\big\vert u(Z_{t_{k}-})\big\vert^{2q}\big]\nonumber\\
&&+C(\Delta t)^{2q}+C\mathbb{E}\big[\sup_{k=1,2,...,n_{T}}\big(\Delta t_{k-1}\big)^{2q}\big\vert u(\bar{Z}_{t_{k}-})\big\vert^{2q}\big].
\end{eqnarray}
We now estimate the two terms on the right-hand side of \eqref{a3.61}. It is easy to show
\begin{eqnarray}\label{3.61}
\!\!&&\!\!\mathbb{E}\big[\sup_{k=1,2,...,n_{T}}\big(\Delta t_{k-1}\big)^{2q}\big\vert u(Z_{t_{k}-})\big\vert^{2q}\big]\nonumber\\
\!\!&\leq&\!\!C\mathbb{E}\big[\sup_{k=1,2,...,n_{T}}\big(\Delta t_{k-1}\big)^{2q}(1+Z_{t_{k}-}^{\frac{\rho+1}{\rho-1}})^{2q}\big]\nonumber\\
\!\!&\leq&\!\!C(\Delta t)^{2q}+C\mathbb{E}\big[\sup_{k=1,2,...,n_{T}}\big(\Delta t_{k-1}\big)^{2q} Z_{t_{k}-}^{\frac{2q(\rho+1)}{\rho-1}}\big].
\end{eqnarray}
Similarly, we have
\begin{eqnarray}
&&\mathbb{E}\big[\sup_{k=1,2,...,n_{T}}\big(\Delta t_{k-1}\big)^{2q}\big\vert u(\bar{Z}_{t_{k}-})\big\vert^{2q}\big]\nonumber\\
&\leq&C(\Delta t)^{2q}+C\mathbb{E}\Big[\sup_{k=1,2,...,n_{T}}\big(\Delta t_{k-1}\big)^{2q}\bar{Z}_{t_{k}-}^{\frac{2q(\rho+1)}{\rho-1}}\Big].
\end{eqnarray}
We observe from \eqref{b17} that 
\begin{eqnarray}\label{0.34}
	\mathbb{E}\Big[\sup_{k=1,2,...,n_{T}}Z_{t_{k}-}^{\frac{2q(\rho+1)}{\rho-1}}\Big]<\infty.
\end{eqnarray}
Then combining the result above and $\eqref{3.28}$ leads to
\begin{eqnarray}\label{0.35}
&&\mathbb{E}\Big[\sup_{k=1,2,...,n_{T}} \bar{Z}_{t_{k}-}^{\frac{2q(\rho+1)}{\rho-1}}\Big]\\
&=&\mathbb{E}\Big[\sup_{k=1,2,...,n_{T}}\big(\bar{Z}_{t_{k}-}-Z_{t_{k}-}+Z_{t_{k}-}\big)^{\frac{2q(\rho+1)}{\rho-1}}\Big]\nonumber\\
&\leq&C\mathbb{E}\Big[\sup_{k=1,2,...,n_{T}}\big\vert \bar{Z}_{t_{k}-}-Z_{t_{k}-}\big\vert^{\frac{2q(\rho+1)}{\rho-1}}\Big]+C\mathbb{E}\Big[\sup_{k=1,2,...,n_{T}} Z_{t_{k}-}^{\frac{2q(\rho+1)}{\rho-1}}\Big]<\infty.\nonumber
\end{eqnarray}
Consequently $\eqref{a3.61}$-$\eqref{0.35}$ yield
\begin{eqnarray}\label{0.36}
&&\mathbb{E}\big[\sup_{k=1,2,...,n_{T}}\big(\Delta t_{k-1}\big)^{2q}\bar{Z}_{t_{k}-}^{-2mq}\big]\nonumber\\
&\leq& C(\Delta t)^{2q}+C\mathbb{E}\big[\sup_{k=1,2,...,n_{T}}\big(\Delta t_{k-1}\big)^{2q} Z_{t_{k}-}^{-2mq}\big].
\end{eqnarray}
 By $\eqref{3.3}$, we have
\begin{eqnarray}
Z_{t_{k}}=Z_{t_{k}-}+\big[\big(Z_{t_{k}-}^{\frac{1}{1-\rho}}+h(Z_{t_{k}-}^{\frac{1}{1-\rho}})\big)^{1-\rho}-Z_{t_{k}-}\big]\Delta N_{k-1}.
\end{eqnarray}
Under Assumption $\ref{ass2.2}$, that is, there exists a constant $r>0$ such that $x+h(x)\geq rx$ for all $x>0$, we deduce
\begin{eqnarray}
	Z_{t_{k}}\leq Z_{t_{k}-}+(r^{1-\rho}-1)\Delta N_{k-1}Z_{t_{k}-}=\big(1+(r^{1-\rho}-1)\Delta N_{k-1}\big)Z_{t_{k}-}.
\end{eqnarray}
Let $\upsilon_{k}:=1+(r^{1-\rho}-1)\Delta N_{k-1}$. Then
\begin{equation}
\upsilon_{k}=\left\{	
\begin{aligned}
	1,& &\Delta N_{k-1}=0,\\
	r^{1-\rho},& &\Delta N_{k-1}=1,
\end{aligned}
\right.
\end{equation}
which yields $\vert Z_{t_{k}-}\vert^{-2mq}\leq \big\vert \frac{1}{\upsilon_{k}}Z_{t_{k}}\big\vert^{-2mq}$. In light of $\eqref{0.36}$, we have the estimate
\begin{eqnarray}\label{0.39}
&&\mathbb{E}\big[\sup_{k=1,2,...,n_{T}}\big(\Delta t_{k-1}\big)^{2q} \bar{Z}_{t_{k}-}^{-2mq}\big]\nonumber\\
&\leq& C(\Delta t)^{2q}+C\mathbb{E}\big[\sup_{k=1,2,...,n_{T}}\big(\Delta t_{k-1}\big)^{2q} Z_{t_{k}}^{-2mq}\big].
\end{eqnarray}
According to the numerical scheme $\eqref{3.4}$,
\begin{eqnarray}\label{z48}
	\bar{Z}_{t_{k}}=\bar{Z}_{t_{k}-}+\big[\big(\bar{Z}_{t_{k}-}^{\frac{1}{1-\rho}}+h(\bar{Z}_{t_{k}-}^{\frac{1}{1-\rho}})\big)^{1-\rho}-\bar{Z}_{t_{k}-}\big]\Delta N_{k-1}.
\end{eqnarray}
In view of Assumption $\ref{ass2.2}$, that is, there exists a constant $\mu>0$ such that $\vert h'(x)\vert\leq \mu$ for all $x>0$, we thereby obtain that 
\begin{eqnarray}
	\vert h(x)\vert\leq C(1+\vert x\vert),
\end{eqnarray}
where constant $C>0$ depends on $\mu$. Consequently, we conclude from \eqref{z48} that
\begin{eqnarray}
\bar{Z}_{t_{k}}\geq \bar{Z}_{t_{k}-}+\big[\big(C+(C+1)\bar{Z}_{t_{k}-}^{\frac{1}{1-\rho}}\big)^{1-\rho}-\bar{Z}_{t_{k}-}\big]\Delta N_{k-1}.
\end{eqnarray}
Since $\big(C+(C+1)\bar{Z}_{t_{k}-}^{\frac{1}{1-\rho}}\big)^{1-\rho}<(\bar{Z}_{t_{k}-}^{\frac{1}{1-\rho}})^{1-\rho}=\bar{Z}_{t_{k}-}$, we have $\bar{Z}_{t_{k}}\geq \big(C+(C+1)\bar{Z}_{t_{k}-}^{\frac{1}{1-\rho}}\big)^{1-\rho}$ and then deduce
\begin{eqnarray}\label{62}
\!\!\bar{Z}_{t_{k}}^{-2mq}\!\leq\! \big\vert \big(C+(C+1)\bar{Z}_{t_{k}-}^{\frac{1}{1-\rho}}\big)^{1-\rho}\big\vert^{-2mq}\!=\!\big\vert C+(C+1)\bar{Z}_{t_{k}-}^{\frac{1}{1-\rho}}\big\vert^{2mq(\rho-1)}.
\end{eqnarray}
Combining \eqref{62} and $\eqref{0.39}$ yields
\begin{eqnarray}
&&\mathbb{E}\big[\sup_{k=1,2,...,n_{T}}\big(\Delta t_{k-1}\big)^{2q}\bar{Z}_{t_{k}}^{-2mq}\big]\nonumber\\
&\leq&C(\Delta t)^{2q}+C\mathbb{E}\big[\sup_{k=1,2,...,n_{T}}\big(\Delta t_{k-1}\big)^{2q} \bar{Z}_{t_{k}-}^{-2mq}\big]\nonumber\\
&\leq& C(\Delta t)^{2q}+C\mathbb{E}\big[\sup_{k=1,2,...,n_{T}}\big(\Delta t_{k-1}\big)^{2q}Z_{t_{k}}^{-2mq}\big]\nonumber.\end{eqnarray}
The proof is completed.
\end{proof}

Let $\mathbf{1}_{G}$ be the indicator function of $G$. According to Lemma \ref{lem3.7} and the standard inequality for the lower tail of the normal distribution, one can prove the following lemma.

\begin{lemma}\label{lem3.9}
Let  Assumptions $\ref{ass2.2}$, $\ref{ass3.3}$ hold and let $\gamma>2\rho-1, \varepsilon\in\big(0,\frac{2(\gamma+1-2\rho)}{3\rho(\gamma-1)}\big)$. $m$ is a constant defined in Lemma \ref{lem3.7}. Assume that $\Delta t$ is sufficiently small such that 
\begin{footnotesize}
\begin{eqnarray}\label{0.65}
  (\Delta t)^{\frac{m\!-\!1}{2m}+\varepsilon}\!\leq\! \frac{\alpha_{2}^{1/m}}{2(\rho\!-\!1)\alpha_{3}^{(m+1)/m}}
\end{eqnarray}
\end{footnotesize}
and 
\begin{footnotesize}
\begin{eqnarray}\label{0.66}
	\!\!\!\!\!\!\!\!(\Delta t)^{\varepsilon}\!\!&<&\!\!\Big(\frac{\alpha_{2}(\rho\!-\!1)}{2(\rho\!-\!1)(\alpha_{-1}\!+\!\alpha_{1})\!+\!2Q_{\gamma,\rho})}\Big)^{\frac{1}{m+1}}\wedge
	\frac{1}{2\!+\!4(\rho\!-\!1)(\alpha_{-1}\!+\!\alpha_{1})\!+\!4Q_{\gamma,\rho}}.
\end{eqnarray}
\end{footnotesize}
 Then for any $q\geq1$, we have
\begin{eqnarray}
\Big\|\sup_{k=1,2,...,n_{T}}\mathbb{E}^{N}\big[\bar{Z}_{t_{k}-}^{-2mq}\mathbf{1}_{\{\bar{Z}_{t_{k}-}\leq(\Delta t)^{\varepsilon}\bar{Z}_{t_{k-1}}\}}\big]\Big\|_{L_{2}(\Omega, \mathbb{R})}<\infty.
\end{eqnarray}
\end{lemma}

A proof is given in the Appendix. Based on the conclusions above, we will show the boundedness of inverse moments of the JABEM at time $T$.

\begin{lemma}\label{lem3.10}
 Let the conditions of Lemma \ref{lem3.9} hold and let  $q\geq1, \varepsilon\in\big(0, \frac{1}{8mq}\wedge \frac{2(\gamma+1-2\rho)}{3\rho(\gamma-1)}\big)$.
 Then we have
\begin{eqnarray}\label{76}
	\mathbb{E}\big[\bar{Z}_{T-}^{-2mq}\big]<\infty.
\end{eqnarray}	
\end{lemma}

\begin{proof}
For any $q\geq 1$, let $\varepsilon\in\big(0,\frac{1}{8mq}\wedge\frac{2(\gamma+1-2\rho)}{3\rho(\gamma-1)}\big)$ and let constant $l_{0}\in(0,1/3]$ be independent of $\Delta t$. By \eqref{62}, we deduce 
\begin{eqnarray}\label{a82}
	\!\!&&\!\!\mathbb{E}^{N}\big[\bar{Z}_{T-}^{-2mq}\mathbf{1}_{\{\Delta t_{n_{T}-1}<l_{0}\Delta t, \bar{Z}_{T-}>(\Delta t)^{\varepsilon}\bar{Z}_{t_{n_{T}-1}}\}}\big]\nonumber\\
	\!\!&\leq&\!\!(\Delta t)^{-2mq\varepsilon} \mathbb{E}^{N}\big[\bar{Z}_{t_{n_{T}-1}}^{-2mq}\mathbf{1}_{\{\Delta t_{n_{T}-1}<l_{0}\Delta t\}}\big]\\
	\!\!&\leq&\!\! C^{q}(\Delta t)^{-2mq\varepsilon}\mathbf{1}_{\{\Delta t_{n_{T}\!-\!1}<l_{0}\Delta t\}}\!+\!C^{q}(\Delta t)^{-2mq\varepsilon}\mathbb{E}^{N}\big[\bar{Z}_{t_{n_{T}\!-\!1}-}^{-2mq}\mathbf{1}_{\{\Delta t_{n_{T}\!-\!1}<l_{0}\Delta t\}}\big].\nonumber
\end{eqnarray}  
Hence
\begin{eqnarray}\label{83}
\!\!&&\!\!\mathbb{E}^{N}\big[\bar{Z}_{T-}^{-2mq}\big]\nonumber\\
\!\!&=&\!\!\mathbb{E}^{N}\big[\bar{Z}_{T-}^{-2mq}\mathbf{1}_{\{\Delta t_{n_{T}-1}\geq l_{0}\Delta t\}}\big]+\mathbb{E}^{N}\big[\bar{Z}_{T-}^{-2mq}\mathbf{1}_{\{\Delta t_{n_{T}-1}<l_{0}\Delta t, \bar{Z}_{T-}\leq(\Delta t)^{\varepsilon}\bar{Z}_{t_{n_{T}-1}}\}}\big]\nonumber\\
\!\!&&\!\!+\mathbb{E}^{N}\big[\bar{Z}_{T-}^{-2mq}\mathbf{1}_{\{\Delta t_{n_{T}-1}<l_{0}\Delta t, \bar{Z}_{T-}>(\Delta t)^{\varepsilon}\bar{Z}_{t_{n_{T}-1}}\}}\big]\\
\!\!&\leq&\!\!\mathbb{E}^{N}\big[\bar{Z}_{T-}^{-2mq}\mathbf{1}_{\{\Delta t_{n_{T}-1}\geq l_{0}\Delta t\}}\big]+\mathbb{E}^{N}\big[\bar{Z}_{T-}^{-2mq}\mathbf{1}_{\{\Delta t_{n_{T}-1}<l_{0}\Delta t, \bar{Z}_{T-}\leq(\Delta t)^{\varepsilon}\bar{Z}_{t_{n_{T}-1}}\}}\big] \nonumber\\
\!\!&&\!\!+C^{q}(\Delta t)^{-2mq\varepsilon}\mathbf{1}_{\{\Delta t_{n_{T}\!-\!1}<l_{0}\Delta t\}}+C^{q}(\Delta t)^{-2mq\varepsilon}\mathbb{E}^{N}\big[\bar{Z}_{t_{n_{T}\!-\!1}-}^{-2mq}\mathbf{1}_{\{\Delta t_{n_{T}\!-\!1}<l_{0}\Delta t\}}\big].\nonumber
\end{eqnarray}
Similarly, the last term on the right-hand side of the inequality in \eqref{83} satisfies
\begin{eqnarray}
	\!\!&&\!\!\mathbb{E}^{N}\big[\bar{Z}_{t_{n_{T}-1}-}^{-2mq}\mathbf{1}_{\{\Delta t_{n_{T}-1}<l_{0}\Delta t\}}\big]\nonumber\\
	\!\!&=&\!\!\mathbb{E}^{N}\big[\bar{Z}_{t_{n_{T}-1}-}^{-2mq}\mathbf{1}_{\{\Delta t_{n_{T}-1}<l_{0}\Delta t,\Delta t_{n_{T}-2}\geq l_{0}(\Delta t)^{1\!+\!\frac{1}{8q}}\}}\big]\nonumber\\
	\!\!&&\!\!+\mathbb{E}^{N}\big[\bar{Z}_{t_{n_{T}-1}-}^{-2mq}\mathbf{1}_{\{\Delta t_{n_{T}-1}<l_{0}\Delta t, \Delta t_{n_{T}-2}<l_{0}(\Delta t)^{1+\frac{1}{8q}},\bar{Z}_{t_{n_{T}-1}-}\leq(\Delta t)^{\varepsilon}\bar{Z}_{t_{n_{T}-2}}\}}\big]\nonumber\\
	\!\!&&\!\!+\mathbb{E}^{N}\big[\bar{Z}_{t_{n_{T}-1}-}^{-2mq}\mathbf{1}_{\{\Delta t_{n_{T}-1}<l_{0}\Delta t, \Delta t_{n_{T}-2}<l_{0}(\Delta t)^{1+\frac{1}{8q}},\bar{Z}_{t_{n_{T}-1}-}>(\Delta t)^{\varepsilon}\bar{Z}_{t_{n_{T}-2}}\}}\big]\nonumber\\
	\!\!&\leq&\!\!\mathbb{E}^{N}\big[\bar{Z}_{t_{n_{T}-1}-}^{-2mq}\mathbf{1}_{\{\Delta t_{n_{T}-1}<l_{0}\Delta t,\Delta t_{n_{T}-2}\geq l_{0}(\Delta t)^{1\!+\!\frac{1}{8q}}\}}\big]\nonumber\\
	\!\!&&\!\!+\mathbb{E}^{N}\big[\bar{Z}_{t_{n_{T}-1}-}^{-2mq}\mathbf{1}_{\{\Delta t_{n_{T}-1}<l_{0}\Delta t, \Delta t_{n_{T}-2}<l_{0}(\Delta t)^{1+\frac{1}{8q}},\bar{Z}_{t_{n_{T}-1}-}\leq(\Delta t)^{\varepsilon}\bar{Z}_{t_{n_{T}-2}}\}}\big]\nonumber\\
	\!\!&&\!\!+C^{q}(\Delta t)^{-2mq\varepsilon}\mathbf{1}_{\{\Delta t_{n_{T}-1}<l_{0}\Delta t, \Delta t_{n_{T}-2}<l_{0}(\Delta t)^{1+\frac{1}{8q}}\}}\nonumber\\
	\!\!&&\!\!+C^{q}(\Delta t)^{-2mq\varepsilon}\mathbb{E}^{N}\big[\bar{Z}_{t_{n_{T}-2}-}^{-2mq}\mathbf{1}_{\{\Delta t_{n_{T}-1}<l_{0}\Delta t, \Delta t_{n_{T}-2}<l_{0}(\Delta t)^{1+\frac{1}{8q}}\}}\big]\nonumber.
\end{eqnarray}
Based on an inductive argument, we show that for any $i\in\{1,...,n_{T}-2\}$, 
\begin{eqnarray}\label{84}
	\!\!&&\!\!\mathbb{E}^{N}\big[\bar{Z}_{t_{n_{T}-i}-}^{-2mq}\mathbf{1}_{\{\Delta t_{n_{T}-1}<l_{0}\Delta t,...,\Delta t_{n_{T}-i}<l_{0}(\Delta t)^{1+\frac{i-1}{8q}}\}}\big]\nonumber\\
	\!\!&\leq&\!\!\mathbb{E}^{N}\big[\bar{Z}_{t_{n_{T}-i}-}^{-2mq}\mathbf{1}_{\{\Delta t_{n_{T}-1}<l_{0}\Delta t,...,\Delta t_{n_{T}-i}<l_{0}(\Delta t)^{1+\frac{i-1}{8q}}, \Delta t_{n_{T}-i-1}\geq l_{0}(\Delta t)^{1+\frac{i}{8q}}\}}\big]\nonumber\\
	\!\!&&\!\!+\mathbb{E}^{N}\big[\bar{Z}_{t_{n_{T}\!-\!i}-}^{-2mq}\mathbf{1}_{\{\Delta t_{n_{T}\!-\!1}<l_{0}\Delta t,...,\Delta t_{n_{T}\!-\!i\!-\!1}< l_{0}(\Delta t)^{1+\frac{i}{8q}},\bar{Z}_{t_{n_{T}\!-\!i}-}\leq(\Delta t)^{\varepsilon}\bar{Z}_{t_{n_{T}\!-\!i\!-\!1}}\}}\big]\nonumber\\
	\!\!&&\!\!+C^q(\Delta t)^{-2mq\varepsilon}\mathbf{1}_{\{\Delta t_{n_{T}-1}<l_{0}\Delta t,...,\Delta t_{n_{T}-i-1}< l_{0}(\Delta t)^{1+\frac{i}{8q}}\}}\nonumber\\
	\!\!&&\!\!+C^q(\Delta t)^{-2mq\varepsilon}\mathbb{E}^{N}\big[\bar{Z}_{t_{n_{T}-i-1}-}^{-2mq}\mathbf{1}_{\{\Delta t_{n_{T}-1}<l_{0}\Delta t,...,\Delta t_{n_{T}-i-1}<l_{0}(\Delta t)^{1+\frac{i}{8q}}\}}\big],
\end{eqnarray}
where the constant $C$ does not depend on $i$. Let $k_{0}:=\min\big\{n\in\{0,1,...,n_{T}\}\colon 0<T-\Delta t<t_n<T\big\}$. If $k_{0}\leq n_{T}-2$, inserting \eqref{84} into \eqref{83} and repeating this procedure for $i=1,...,n_{T}\!-\!k_{0}\!-\!1$ yield
\begin{eqnarray}\label{85}
		\!\!&&\!\!\mathbb{E}^{N}\big[\bar{Z}_{T-}^{-2mq}\big]\nonumber\\
		\!\!&\leq&\!\!\mathbb{E}^{N}\big[\bar{Z}_{T-}^{-2mq}\mathbf{1}_{\{\Delta t_{n_{T}\!-\!1}\geq l_{0}\Delta t\}}\big]\!+\!\mathbb{E}^{N}\big[\bar{Z}_{T-}^{-2mq}\mathbf{1}_{\{\Delta t_{n_{T}\!-\!1}<l_{0}\Delta t, \bar{Z}_{T-}\leq(\Delta t)^{\varepsilon}\bar{Z}_{t_{n_{T}\!-\!1}}\}}\big] \nonumber\\
		\!\!&&\!\!+\sum_{i=1}^{n_{T}-k_{0}-1}\big(C^q(\Delta t)^{-2mq\varepsilon}\big)^{i}\mathbb{E}^{N}\Big[\bar{Z}_{t_{n_{T}-i}-}^{-2mq}\nonumber\\
		\!\!&&\!\!\cdot\mathbf{1}_{\{\Delta t_{n_{T}-1}<l_{0}\Delta t,...,\Delta t_{n_{T}-i}<l_{0}(\Delta t)^{1+\frac{i-1}{8q}},\Delta t_{n_{T}-i-1}\geq l_{0}(\Delta t)^{1+\frac{i}{8q}}\}}\Big]\nonumber\\
		\!\!&&\!\!+\sum_{i=1}^{n_{T}-k_{0}-1}\big(C^q(\Delta t)^{-2mq\varepsilon}\big)^{i}\mathbb{E}^{N}\Big[\bar{Z}_{t_{n_{T}-i}-}^{-2mq}\nonumber\\
		\!\!&&\!\!\cdot\mathbf{1}_{\{\Delta t_{n_{T}\!-\!1}<l_{0}\Delta t,...,\Delta t_{n_{T}-i-1}<l_{0}(\Delta t)^{1+\frac{i}{8q}},\bar{Z}_{t_{n_{T}-i}-}\leq(\Delta t)^{\varepsilon}\bar{Z}_{t_{n_{T}-i-1}}\}}\Big]\nonumber\\
		\!\!&&\!\!+\sum_{i=1}^{n_{T}-k_{0}}\big(C^q(\Delta t)^{-2mq\varepsilon}\big)^{i}\mathbf{1}_{\{\Delta t_{n_{T}-1}<l_{0}\Delta t,...,\Delta t_{n_{T}-i}<l_{0}(\Delta t)^{1+\frac{i\!-\!1}{8q}}\}}\\
		\!\!&&\!\!+\big(C^q(\Delta t)^{-2mq\varepsilon}\big)^{n_{T}\!-k_{0}}\mathbb{E}^{N}\Big[\bar{Z}_{t_{k_{0}}-}^{-2mq}\mathbf{1}_{\{\Delta t_{n_{T}\!-\!1}<l_{0}\Delta t,...,\Delta t_{k_{0}}<l_{0}(\Delta t)^{1+\frac{n_{T}\!-\!k_{0}\!-\!1}{8q}}\}}\Big].\nonumber
	\end{eqnarray}
In view of $l_{0}\leq \frac{1}{3}, \varepsilon<\frac{1}{8mq}, m>1$ and \eqref{0.66}, we have $(\Delta t)^{\frac{1}{8q}}<(\Delta t)^{\varepsilon}<\frac{1}{2}$. In addition, if $\Delta t_{n_{T}-1}<l_{0}\Delta t,...,\Delta t_{k_{0}}<l_{0}(\Delta t)^{1+\frac{n_{T}-k_{0}-1}{8q}}$, then we deduce 
$$\sum_{j=1}^{n_{T}-k_{0}}\Delta t_{n_{T}-j}<l_{0}\sum_{j=1}^{n_{T}-k_{0}}(\Delta t)^{1+\frac{j-1}{8q}}<\frac{l_{0}\Delta t}{1-(\Delta t)^{\frac{1}{8q}}}<2l_{0}\Delta t$$
and $\Delta t_{k_{0}-1}=\Delta t-\sum_{j=1}^{n_{T}-k_{0}}\Delta t_{n_{T}-j}> l_{0}\Delta t$. Thus for the last term on the right-hand side of \eqref{85}, we obtain
 \begin{eqnarray}\label{86}
\!\!&&\!\!\mathbb{E}^{N}\Big[\bar{Z}_{t_{k_{0}}-}^{-2mq}\mathbf{1}_{\{\Delta t_{n_{T}-1}<l_{0}\Delta t,...,\Delta t_{k_{0}}<l_{0}(\Delta t)^{1+\frac{n_{T}-k_{0}-1}{8q}}\}}\Big]\nonumber\\
\!\!&\leq&\!\! \mathbb{E}^{N}\Big[\bar{Z}_{t_{k_{0}}-}^{-2mq}\mathbf{1}_{\{\Delta t_{k_{0}-1}=\Delta t-\sum_{j=1}^{n_{T}-k_{0}}\Delta t_{n_{T}-j}> l_{0}\Delta t\}}\Big]\nonumber\\
\!\!&\leq&\!\! l_{0}^{-2q}\mathbb{E}^{N}\Big[\bar{Z}_{t_{k_{0}}-}^{-2mq}\big(\frac{\Delta t_{k_{0}-1}}{\Delta t}\big)^{2q}\Big].
\end{eqnarray}
Inserting \eqref{86} into \eqref{85} leads to
\begin{eqnarray}\label{87}
	\!\!&&\!\!\mathbb{E}^{N}\big[\bar{Z}_{T-}^{-2mq}\big]\nonumber\\
	\!\!&\leq&\!\!\sum_{i=0}^{n_{T}-k_{0}-1}\big(C^q(\Delta t)^{-2mq\varepsilon}\big)^{i}\mathbb{E}^{N}\big[\bar{Z}_{t_{n_{T}-i}-}^{-2mq}\big(\frac{\Delta t_{n_{T}-i-1}}{\Delta t}\big)^{2q}\big]\frac{1}{l_{0}^{2q}(\Delta t)^{\frac{i}{4}}}\nonumber\\
		\!\!&&\!\!+\sum_{i=0}^{n_{T}-k_{0}-1}\big(C^q(\Delta t)^{-2mq\varepsilon}\big)^{i}\mathbb{E}^{N}\Big[\bar{Z}_{t_{n_{T}-i}-}^{-2mq}\mathbf{1}_{\{\bar{Z}_{t_{n_{T}-i}-}\leq(\Delta t)^{\varepsilon}\bar{Z}_{t_{n_{T}-i-1}}\}}\Big]\nonumber\\
		\!\!&&\!\!+\!\!\sum_{i=1}^{n_{T}\!-\!k_{0}}\!\!\!\big(C^q(\Delta t)^{\!-2mq\varepsilon}\big)^{\!i}\!+\!l_{0}^{-2q}\big(C^q(\Delta t)^{\!-2mq\varepsilon}\big)^{n_{T}\!-\!k_{0}}\mathbb{E}^{N}\Big[\bar{Z}_{t_{k_{0}}-}^{-2mq}\big(\frac{\Delta t_{k_{0}\!-\!1}}{\Delta t}\big)^{2q}\Big].
\end{eqnarray}
Since the number of time nodes on the interval $(T-\Delta t, T)$ does not exceed $N_{\Delta t}$, it implies that $n_{T}-k_{0}\leq N_{\Delta t}$. Therefore, we infer 
\begin{eqnarray}
	\!\!&&\!\!\mathbb{E}^{N}\big[\bar{Z}_{T-}^{-2mq}\big]\nonumber\\
	\!\!&\leq&\!\! \sum_{i=1}^{N_{\Delta t}}\big(C^q(\Delta t)^{-2mq\varepsilon}\big)^{i}+l_{0}^{-2q}\sup_{k=1,2,...,n_T}\mathbb{E}^{N}\big[\bar{Z}_{t_{k}-}^{-2mq}\big(\frac{\Delta t_{k-1}}{\Delta t}\big)^{2q}\big]\nonumber\\
	\!\!&&\!\!\cdot\bigg(\sum_{i=0}^{N_{\Delta t}-1}\big(C^q(\Delta t)^{-(\frac{1}{4}+2mq\varepsilon)}\big)^{i}+\big(C^q(\Delta t)^{-2mq\varepsilon}\big)^{N_{\Delta t}}\bigg)\nonumber\\
	\!\!&&\!\!+\sup_{k=1,2,...,n_T}\mathbb{E}^{N}\Big[\bar{Z}_{t_{k}-}^{-2mq}\mathbf{1}_{\{\bar{Z}_{t_{k}-}\leq(\Delta t)^{\varepsilon}\bar{Z}_{t_{k-1}}\}}\Big]\sum_{i=0}^{N_{\Delta t}-1}\big(C^q(\Delta t)^{-2mq\varepsilon}\big)^{i}\nonumber\\
	\!\!&\leq&\!\! C^q(\Delta t)^{-2mq\varepsilon}\frac{\big(C^q(\Delta t)^{-2mq\varepsilon}\big)^{N_{\Delta t}}-1}{C^q(\Delta t)^{-2mq\varepsilon}-1}+l_{0}^{-2q}\sup_{k=1,2,...,n_T}\mathbb{E}^{N}\big[\bar{Z}_{t_{k}-}^{-2mq}\big(\frac{\Delta t_{k-1}}{\Delta t}\big)^{2q}\big]\nonumber\\
	\!\!&&\!\!\cdot\bigg(\frac{\big(C^q(\Delta t)^{-(\frac{1}{4}+2mq\varepsilon)}\big)^{N_{\Delta t}}-1}{C^q(\Delta t)^{-(\frac{1}{4}+2mq\varepsilon)}-1}+\big(C^q(\Delta t)^{-2mq\varepsilon}\big)^{N_{\Delta t}}\bigg)\nonumber\\
	\!\!&&\!\!+\sup_{k=1,2,...,n_T}\mathbb{E}^{N}\Big[\bar{Z}_{t_{k}-}^{-2mq}\mathbf{1}_{\{\bar{Z}_{t_{k}-}\leq(\Delta t)^{\varepsilon}\bar{Z}_{t_{k-1}}\}}\Big]\frac{\big(C^q(\Delta t)^{-2mq\varepsilon}\big)^{N_{\Delta t}}-1}{C^q(\Delta t)^{-2mq\varepsilon}-1}\nonumber\\
	\!\!&\leq&\!\!2\big(C^q(\Delta t)^{-2mq\varepsilon}\big)^{N_{\Delta t}} \nonumber\\
	\!\!&&\!\!+2l_{0}^{-2q}\sup_{k=1,2,...,n_T}\mathbb{E}^{N}\big[\bar{Z}_{t_{k}-}^{-2mq}\big(\frac{\Delta t_{k-1}}{\Delta t}\big)^{2q}\big]\big(C^q(\Delta t)^{-(\frac{1}{4}+2mq\varepsilon)}\big)^{N_{\Delta t}}\nonumber\\
	\!\!&&\!\!+\sup_{k=1,2,...,n_T}\mathbb{E}^{N}\Big[\bar{Z}_{t_{k}-}^{-2mq}\mathbf{1}_{\{\bar{Z}_{t_{k}-}\leq(\Delta t)^{\varepsilon}\bar{Z}_{t_{k-1}}\}}\Big]\big(C^q(\Delta t)^{-2mq\varepsilon}\big)^{N_{\Delta t}},
\end{eqnarray}
 where constant $C$ is independent of $\Delta t$ and $C^q>2$. With the aid of the properties of the Poisson process and the condition $\varepsilon<\frac{1}{8mq}$, we arrive at
\begin{eqnarray}\label{d62}
	\!\!&&\!\!\mathbb{E}\Big[\big(C^{2q}(\Delta t)^{-(\frac{1}{2}+4mq\varepsilon)}\big)^{N_{\Delta t}}\Big]\nonumber\\
	\!\!&=&\!\!\sum_{j=0}^{\infty}\big(C^{2q}(\Delta t)^{-(\frac{1}{2}+4mq\varepsilon)}\big)^{j}\frac{(\lambda \Delta t)^{j}}{j!}\exp(-\lambda \Delta t)\nonumber\\
	\!\!&\leq&\!\!\sum_{j=0}^{\infty}\frac{(C^{2q})^j\lambda^{j}}{j!}\exp(-\lambda \Delta t)<\infty.
\end{eqnarray}
Consequently, with sufficiently small $\Delta t$ satisfying \eqref{0.65} and \eqref{0.66}, by the H\"{o}lder inequality, it follows from \eqref{d62}, \eqref{c83} and \eqref{x99} that
\begin{eqnarray}
\!\!&&\!\!\mathbb{E}\big[\bar{Z}_{T-}^{-2mq}\big]\nonumber\\
\!\!&\leq&\!\!2\mathbb{E}\Big[\big(C^q(\Delta t)^{-2mq\varepsilon}\big)^{N_{\Delta t}}\Big]+2l_{0}^{-2q}\Big(\mathbb{E}\Big[\big(C^{2q}(\Delta t)^{-(\frac{1}{2}+4mq\varepsilon)}\big)^{N_{\Delta t}}\Big]\Big)^{\frac{1}{2}}\nonumber\\
\!\!&&\!\!\cdot\Big(\mathbb{E}\Big[\sup_{k=1,2,...,n_{T}}\bar{Z}_{t_{k}-}^{-4mq}\big(\frac{\Delta t_{k-1}}{\Delta t}\big)^{4q}\Big]\Big)^{\frac{1}{2}}+\Big(\mathbb{E}\Big[\big(C^{2q}(\Delta t)^{-4mq\varepsilon}\big)^{N_{\Delta t}}\Big]\Big)^{\frac{1}{2}}\nonumber\\
\!\!&&\!\!\cdot\Big\|\sup_{k=1,2,...,n_{T}}\mathbb{E}^{N}\big[\bar{Z}_{t_{k}-}^{-2mq}\mathbf{1}_{\{\bar{Z}_{t_{k}-}\leq(\Delta t)^{\varepsilon}\bar{Z}_{t_{k-1}}\}}\big]\Big\| _{L_{2}(\Omega, \mathbb{R})}\nonumber\\
\!\!&<&\!\!\infty.
\end{eqnarray}
The proof is completed.
\end{proof}

In the last part of this section, we will present the main convergence result of the numerical method (TJABEM) for $\eqref{1.1}$ under appropriate conditions. Meanwhile, we find that the numerical scheme preserves the positivity of the original model $\eqref{1.1}$ after transforming back.

\subsection{Main result}
To transform the numerical method $\bar{Z}_{t_{k}}$ in $\eqref{3.4}$ back, we define the numerical method as
\begin{eqnarray}\label{4.1}
	\bar{X}_{t_{k}}=(\bar{Z}_{t_{k}})^{\frac{1}{1-\rho}},\ \ k\in\{0,1,...,n_{T}\},
\end{eqnarray}
  called the TJABEM for the original model $\eqref{1.1}$. Since $\bar{Z}_{t_{k}}>0$ a.s. for $k\in\{0,1,...,n_{T}\}$, the numerical solution $\{\bar{X}_{t_{k}}\}_{k\in\{0,1,...,n_{T}\}}$ is inside the positive domain obviously. 

\begin{theorem}\label{thm4.1}
Let Assumptions $\ref{ass2.2}$, $\ref{ass3.3}$ hold and let $\gamma> 2\rho-1$. Then for any $p\geq 1$ and $\varepsilon\in\big(0,\frac{\rho-1}{8\rho p}\wedge\frac{2(\gamma+1-2\rho)}{3\rho(\gamma-1)}\big)$, with sufficiently small $\Delta t$ satisfying \eqref{0.65} and \eqref{0.66}, we have
\begin{eqnarray}\label{82}
	\mathbb{E}\big[\big\vert\bar{X}_{T}-X_{T}\big\vert^{p}\big]\leq C(\Delta t)^{p}.
\end{eqnarray}
\end{theorem}

\begin{proof}
	Since $X_{T}=Z_{T}^{\frac{1}{1-\rho}}$ and $\bar{X}_{T}=\bar{Z}_{T}^{\frac{1}{1-\rho}}$, the mean value theorem gives
	\begin{eqnarray}
	\big\vert \bar{X}_{T}-X_{T}\big\vert=\big\vert \bar{Z}_{T}^{\frac{1}{1-\rho}}-Z_{T}^{\frac{1}{1-\rho}}\big\vert\leq \frac{1}{\rho-1}\big\vert \bar{Z}_{T}-Z_{T}\big\vert\cdot \big\vert\bar{Z}_{T}\wedge Z_{T}\big\vert^{\frac{\rho}{1-\rho}}.
	\end{eqnarray}
The H\"{o}lder inequality and Theorem $\ref{thm3.5}$ ensure that for any $p\geq 1$, 
\begin{eqnarray}\label{0.45}
\mathbb{E}\big[\big\vert\bar{X}_{T}-X_{T}\big\vert^{p}\big]&=& \mathbb{E}\Big[\big\vert\bar{Z}_{T}^{\frac{1}{1-\rho}}-Z_{T}^{\frac{1}{1-\rho}}\big\vert^{p}\Big]\nonumber\\
&\leq&C\mathbb{E}\Big[\big\vert\bar{Z}_{T}-Z_{T}\big\vert^{p}\big\vert\bar{Z}_{T}\wedge Z_{T}\big\vert^{\frac{p\rho}{1-\rho}}\Big]\nonumber\\
&\leq&C\Big(\mathbb{E}\Big[\big\vert\bar{Z}_{T}-Z_{T}\big\vert^{2p}\Big]\Big)^{1/2}\Big(\mathbb{E}\Big[\big\vert\bar{Z}_{T}\wedge Z_{T}\big\vert^{\frac{2p\rho}{1-\rho}}\Big]\Big)^{1/2}\nonumber\\
&\leq& C(\Delta t)^{p}\Big(\mathbb{E}\Big[\big\vert\bar{Z}_{T}\wedge Z_{T}\big\vert^{\frac{2p\rho}{1-\rho}}\Big]\Big)^{1/2}.
\end{eqnarray}
Since $\bar{Z}_{T}>0$ a.s., for $\gamma>2\rho-1$, we have
\begin{eqnarray}
\mathbb{E}\Big[\bar{Z}_{T}^{\frac{2p\rho}{1-\rho}}\Big]\leq C+\mathbb{E}\Big[\bar{Z}_{T}^{\frac{2p\rho}{1-\rho}}\mathbf{1}_{\{\bar{Z}_{T}<1\}}\Big]\leq C+\mathbb{E}\Big[\bar{Z}_{T}^{-2m\frac{\rho p}{\gamma-\rho}}\Big],
\end{eqnarray}
where $m=\frac{\gamma-\rho}{\rho-1}$. Corollary $\ref{cor3.2}$, \eqref{62} and Lemma $\ref{lem3.10}$ lead to
\begin{eqnarray}\label{0.47}
\mathbb{E}\Big[\big\vert\bar{Z}_{T}\wedge Z_{T}\big\vert^{\frac{2p\rho}{1-\rho}}\Big]\leq C\!+\!\mathbb{E}\Big[\bar{Z}_{T}^{-2m\frac{\rho p}{\gamma-\rho}}\Big]\!+\!\mathbb{E}\Big[Z_{T}^{\frac{2p\rho}{1-\rho}}\Big]\leq C\!+\!C\mathbb{E}\Big[\bar{Z}_{T-}^{-2m\frac{\rho p}{\gamma-\rho}}\Big]<\infty.
\end{eqnarray}
Inserting $\eqref{0.47}$ into $\eqref{0.45}$ yields the desired assertion.
\end{proof}

\section{Numerical experiments}
Consider Ait--Sahalia-type interest rate model with Poisson jumps
 \begin{eqnarray}\label{5.1}
 	dX_{t}&=&(\alpha_{-1}X_{t-}^{-1}-\alpha_{0}+\alpha_{1}X_{t-}-\alpha_{2}X_{t-}^{\gamma})dt\nonumber\\
 	&&+\alpha_{3} X_{t-}^{\rho}dW_t+h(X_{t-})dN_t,\ \  t\in(0,T], \ \ X_{0}=x_{0},
 \end{eqnarray}
with $\alpha_{-1},\alpha_{0},\alpha_{1},\alpha_{2},\alpha_{3}>0$ and $\gamma>2\rho-1, \rho>1$. We conduct simulations using of our numerical method. In our experiments, we choose two sets of parameters:
\begin{itemize}
	\item[I.] $\alpha_{-1}=2,\alpha_{0}=1,\alpha_{1}=1.5,\alpha_{2}=5, \alpha_{3}=1, \gamma=3, \rho=1.5, x_{0}=1$; \label{1}
	\item[II.] $\alpha_{-1}=1,\alpha_{0}=2,\alpha_{1}=1.5,\alpha_{2}=3, \alpha_{3}=1, \gamma=3.5, \rho=1.5, x_{0}=1$.\label{2}
\end{itemize}

Firstly, we present the percentages of non-positive numerical values in Tab.\ref{tab1} for the two cases above with different jump coefficients $h(x)$ where 5000 sample trajectories are simulated. From which we can see that the numerical method (TJABEM) is positivity preserving.

\begin{table}[h]
\begin{center}
\begin{minipage}{\textwidth}
\caption{Percentages of non-positive numerical values of the TJABEM with $\lambda=1$, the different jump coefficients $h(x)$ and the different parameters sets.}\label{tab1}%
\resizebox{\textwidth}{18mm}{
\begin{tabular}{@{}cccc@{}}
\toprule
\textbf{Step sizes(T=1)} & \textbf{h(x)}  & \textbf{TJABEM($\%$)( Case I.)} & \textbf{TJABEM($\%$)(Case II.)}\\
\midrule	
	              & -0.5x  & 0  & 0 \\
	$2^{-5}$ &  0.5x  & 0  & 0  \\
			      & sin(x) & 0  & 0  \\
		          &   &   &     \\ [-9pt]
			      	\hline
			      &   &   &     \\ [-9pt]
	              & -0.5x  & 0  & 0 \\
   $2^{-6}$ &  0.5x  & 0  & 0 \\
			      & sin(x) & 0  & 0 \\
			      &   &   &     \\ [-9pt]
			      	\hline
			      &   &   &     \\ [-9pt]
                  &-0.5x   & 0  & 0 \\
   $2^{-7}$ &  0.5x  & 0  & 0 \\
			      & sin(x) & 0  & 0 \\
\botrule
\end{tabular}}
\end{minipage}
\end{center}
\end{table}

Secondly, we aim to verify that our numerical method is strongly convergent with order one. Normally, we measure the approximation errors in terms of $\varepsilon(\Delta t):=\mathbb{E}\big[\vert X_{T}-\bar{X}_{T}\vert\big]$. In the corresponding plots we show $\log(\varepsilon(\Delta t))$ of the strong error versus $\log(\Delta t)$ of the maximal time step size. The slopes of the estimated error in the log-log plots indicate the order of the strong convergence attained by our numerical method. Here the expectation is approximated by the Monte Carlo method, using 5000 Brownian and Poisson paths. 

\begin{figure}[!ht]\centering
\begin{tabular}{cc}
\includegraphics[height=5cm,width=5.8cm]{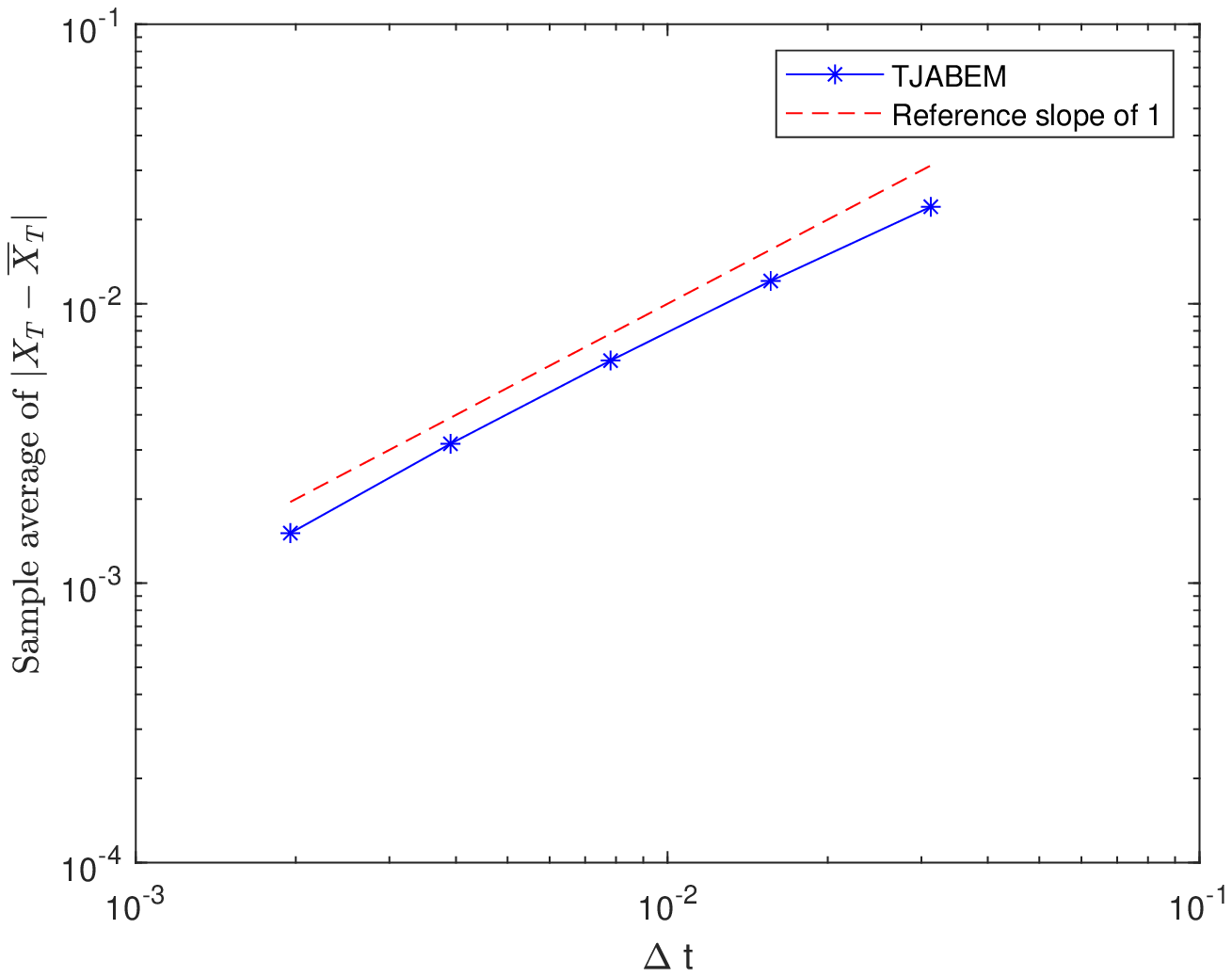} &
\includegraphics[height=5cm,width=5.8cm]{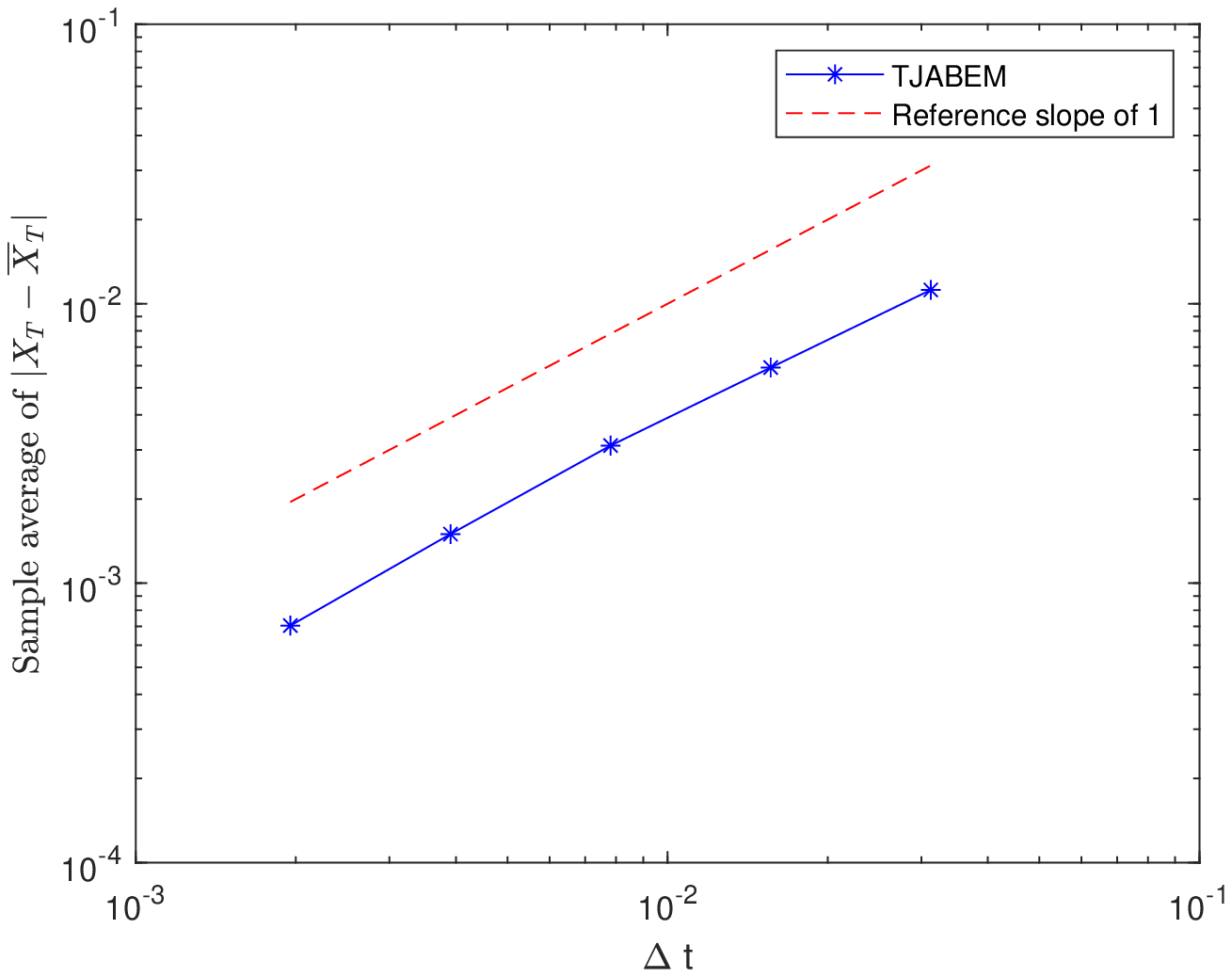} \\
\end{tabular}
\caption{Numerical simulations of \eqref{5.1} with parameters set $I.$, $h(x)=-0.5x, \lambda=1$ (left) and parameters set $II.$, $h(x)=-0.5x, \lambda=1$ (right).}
\label{fig1}
\end{figure}

\begin{figure}[!ht]\centering
\begin{tabular}{cc}
\includegraphics[height=5cm,width=5.8cm]{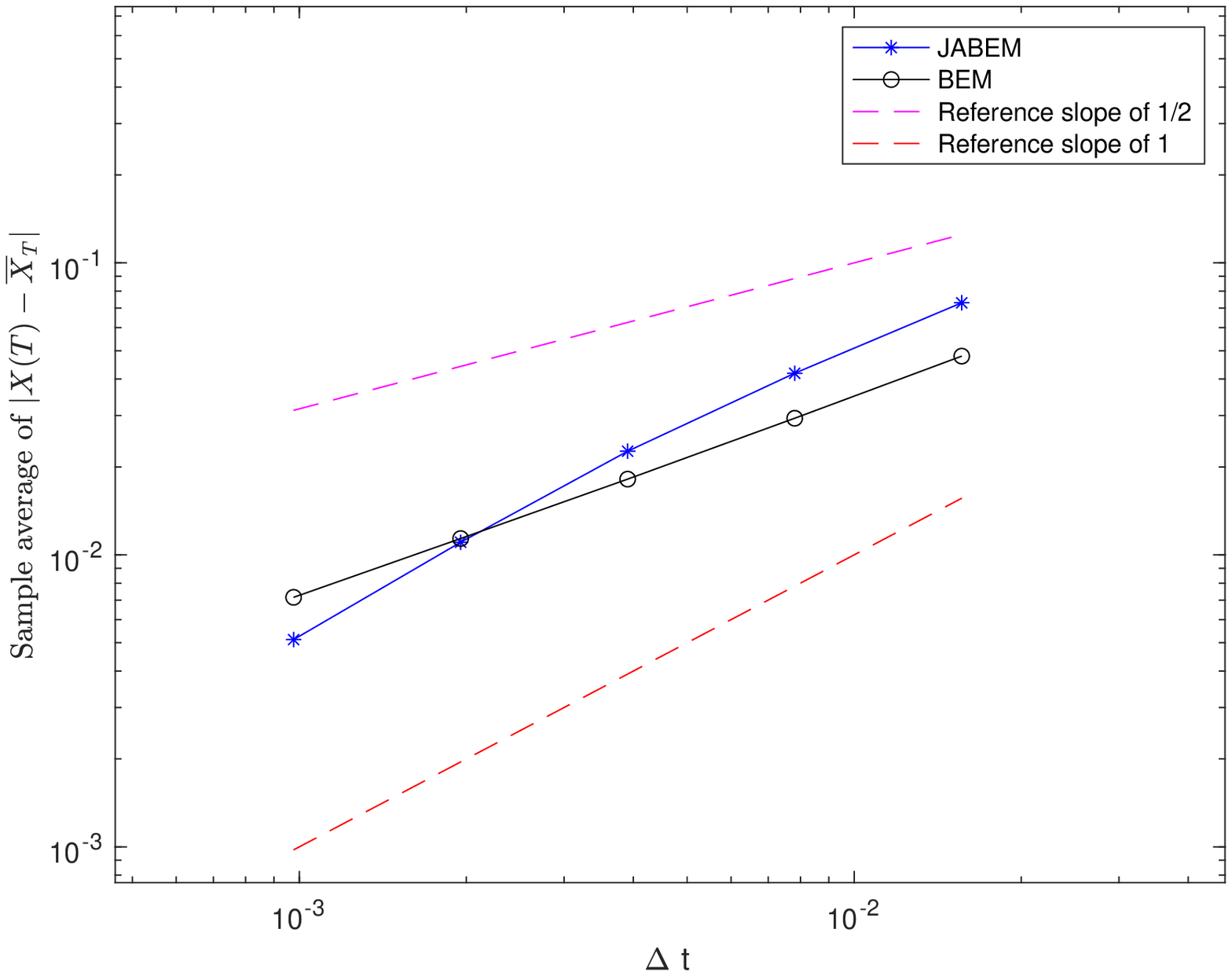} &
\includegraphics[height=5cm,width=5.8cm]{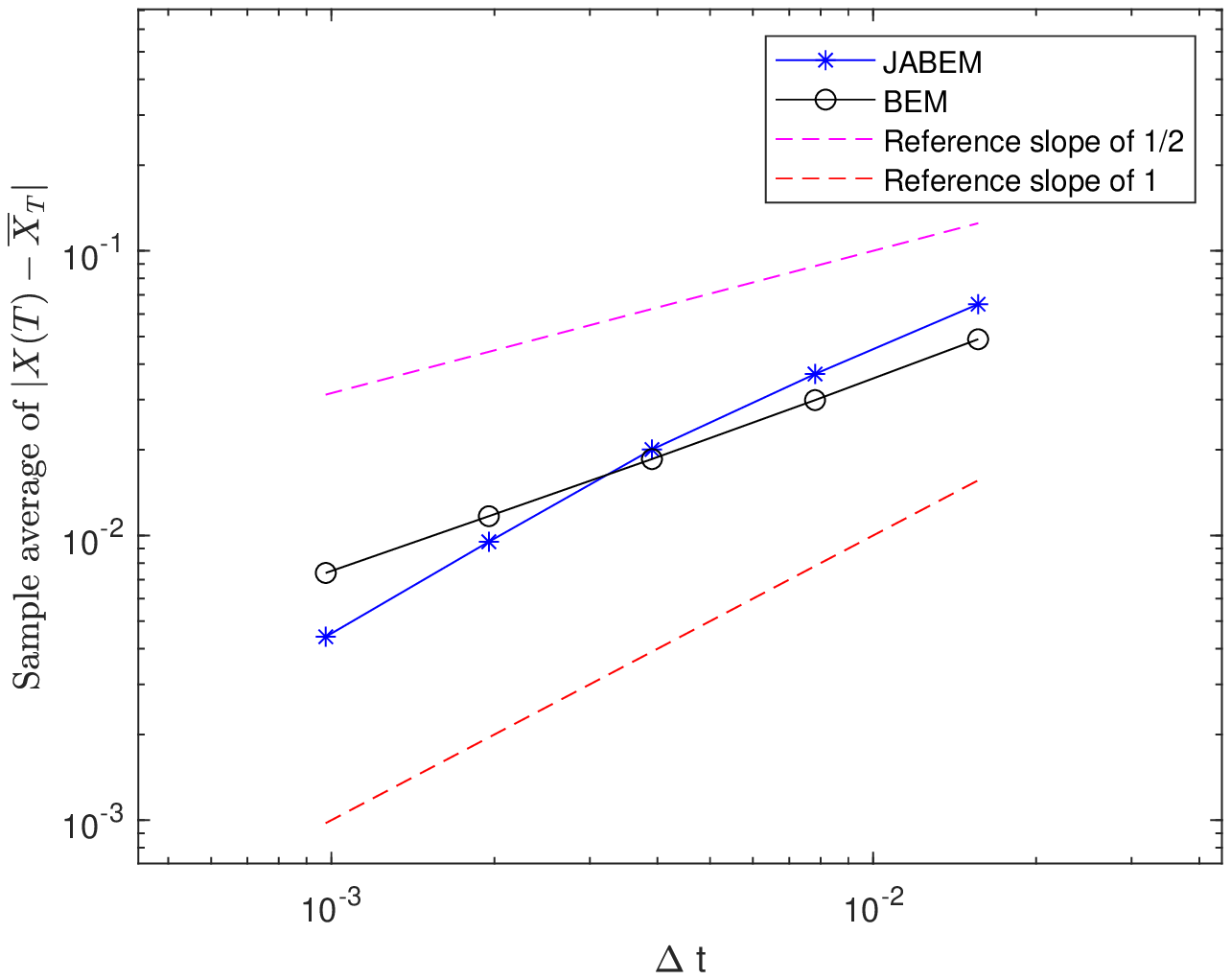} \\
\end{tabular}
\caption{Numerical simulations of \eqref{5.1} with parameters set $I.$, $h(x)=x, \lambda=5$ (left) and parameters set $II.$, $h(x)=x,\lambda=5$ (right).}
\label{fig2}
\end{figure} 

Looking closely at the figures above, we draw the following conclusions under different intensities,
\begin{itemize}
\item when the intensity $\lambda=1$, the exact solution of the model $\eqref{1.1}$ is identified with the numerical one using a small step size $\Delta t_{exact}=2^{-12}$. For five different step sizes $\Delta t=2^{-i}, i=5,...,9$, as predicted, the slopes of the errors (solid lines) and the reference (dashed line) match well, which indicates that the proposed method shows a strong convergence rate of order one. Therefore, the results showed in Fig.\ref{fig1} are consistent with our theoretical results, namely Theorem $\ref{thm4.1}$.

	\item when the intensity $\lambda=5$, the exact solution of the model $\eqref{1.1}$ is identified with the numerical one using a small step size $\Delta t_{exact}=2^{-13}$. For five different step sizes $\Delta t=2^{-i}, i=6,...,10$, we compare the convergence rate of the TJABEM with that of the BEM in \cite{Zhao2021On}. As shown in Fig.\ref{fig2}, the TJABEM has a strong convergence rate of order one. At the same time, one can see that the BEM is strongly convergent with order only one half. 
\end{itemize}

\section*{Appendix}\label{App}
\subsection*{Proof of Lemma \ref{lem3.6}.} 
\begin{proof}
For $k\in\{1,2,...,n_{T}\}$, let $\Phi(x,t_{k-1},\Delta t_{k-1})$ denote the numerical value at time $t_{k}-$ where the initial value is $x$ at time $t_{k-1}$, that is,
\begin{eqnarray}\label{0.2}
	\Phi(x,t_{k-1},\Delta t_{k-1})=x+F_{\gamma,\rho}(\Phi(x,t_{k-1},\Delta t_{k-1}))\Delta t_{k-1}+(1-\rho)\alpha_{3}\Delta W_{k-1}.
\end{eqnarray}
Obviously, $\bar{Z}_{t_{k}-}=\Phi(\bar{Z}_{t_{k-1}},t_{k-1},\Delta t_{k-1})$. Moreover, we write $Z^{t_{k-1},x}_{t_{k}-}$ in lieu of the value of exact solution at time $t_{k}-$ to highlight the initial value $Z_{t_{k-1}}=x$, that is,
\begin{eqnarray}\label{0.3}
	Z^{t_{k-1},x}_{t_{k}-}=x+\int_{t_{k-1}}^{t_{k}}F_{\gamma,\rho}(Z^{t_{k-1},x}_{t-})\,dt+\int_{t_{k-1}}^{t_{k}}(1-\rho)\alpha_{3}\,dW_t.
\end{eqnarray}
We have $Z_{0}=\bar{Z}_{0}$ for $k=0$. For any $k\in\{1,...,n_{T}\}$,
\begin{eqnarray}
&&F_{\gamma,\rho}(Z_{t_{k}-})-F_{\gamma,\rho}(\bar{Z}_{t_{k}-})\nonumber\\
&=&F_{\gamma,\rho}(Z^{t_{k-1},Z_{t_{k-1}}}_{t_{k}-})-F_{\gamma,\rho}\big(\Phi(\bar{Z}_{t_{k-1}},t_{k-1},\Delta t_{k-1})\big)\nonumber\\
&=&\big[F_{\gamma,\rho}(Z^{t_{k-1},Z_{t_{k-1}}}_{t_{k}-})-F_{\gamma,\rho}\big(\Phi(Z_{t_{k-1}},t_{k-1},\Delta t_{k-1})\big)\big]\nonumber\\
&&+\big[F_{\gamma,\rho}\big(\Phi(Z_{t_{k-1}},t_{k-1},\Delta t_{k-1})\big)-F_{\gamma,\rho}\big(\Phi(\bar{Z}_{t_{k-1}},t_{k-1},\Delta t_{k-1})\big)\big]\nonumber\\
&=:&J_{1}+J_{2}.
\end{eqnarray}
 Under the condition $q\geq 1$, we infer
\begin{eqnarray}
&&\mathbb{E}\Big[\sup_{k=1,2,...,n_{T}}\big\vert\Delta t_{k-1}(F_{\gamma,\rho}(Z_{t_{k}-})-F_{\gamma,\rho}(\bar{Z}_{t_{k}-}))\big\vert^{2q}\Big]\nonumber\\
&\leq& C\bigg((\Delta t)^{2q}\mathbb{E}\Big[\sup_{k=1,2,...,n_{T}}\vert J_{1}\vert^{2q}\Big]+\mathbb{E}\Big[\sup_{k=1,2,...,n_{T}}\vert \Delta t_{k-1}J_{2}\vert^{2q}\Big]\bigg).
\end{eqnarray}
We now estimate the moment of $J_{1}$,
\begin{eqnarray}
&&\mathbb{E}\Big[\sup_{k=1,2,...,n_{T}}\vert J_{1}\vert^{2q}\Big]\nonumber\\
\!\!&=&\!\!\mathbb{E}\Big[\!\sup_{k=1,2,...,n_{T}}\!\!\Big\vert \!\int_{0}^{1}\!\!F'_{\gamma,\rho}\big(Z^{t_{k-1},Z_{t_{k-1}}}_{t_{k}-}\!\!+\!\theta(\Phi(Z_{t_{k-1}},t_{k-1},\Delta t_{k-1})\!-\!Z^{t_{k-1},Z_{t_{k-1}}}_{t_{k}-})\big)\!\,d\theta\nonumber\\
&&\cdot\big(Z^{t_{k-1},Z_{t_{k-1}}}_{t_{k}-}-\Phi(Z_{t_{k-1}},t_{k-1},\Delta t_{k-1})\big)\Big\vert^{2q}\Big].
\end{eqnarray}
Let
$$M_{\gamma,\rho,t_{k}}:=\int_{0}^{1}F'_{\gamma,\rho}(Z^{t_{k-1},Z_{t_{k-1}}}_{t_{k}-}+\theta(\Phi(Z_{t_{k-1}},t_{k-1},\Delta t_{k-1})-Z^{t_{k-1},Z_{t_{k-1}}}_{t_{k}-}))\,d\theta.$$
By $\eqref{3.6}$, we have $M_{\gamma,\rho,t_{k}}\leq Q_{\gamma,\rho}$, a.s. In light of $\eqref{3.17}$, we see
\begin{eqnarray}
&&Z^{t_{k-1},Z_{t_{k-1}}}_{t_{k}-}-\Phi(Z_{t_{k-1}},t_{k-1},\Delta t_{k-1})\nonumber\\
&=&\big[F_{\gamma,\rho}(Z^{t_{k-1},Z_{t_{k-1}}}_{t_{k}-})-F_{\gamma,\rho}\big(\Phi(Z_{t_{k-1}},t_{k-1},\Delta t_{k-1})\big)\big]\Delta t_{k-1}-\Xi_{k}\nonumber\\
&=&M_{\gamma,\rho,t_{k}}\Delta t_{k-1}\big(Z^{t_{k-1},Z_{t_{k-1}}}_{t_{k}-}-\Phi(Z_{t_{k-1}},t_{k-1},\Delta t_{k-1})\big)-\Xi_{k}\nonumber.
\end{eqnarray}
Since $M_{\gamma,\rho,t_{k}}\Delta t_{k-1}\leq Q_{\gamma,\rho}\Delta t<\xi$ a.s. where $\xi\in(0,1)$, we deduce
\begin{eqnarray}\label{3.37}
	Z^{t_{k-1},Z_{t_{k-1}}}_{t_{k}-}-\Phi(Z_{t_{k-1}},t_{k-1},\Delta t_{k-1})=\frac{\Xi_{k}}{M_{\gamma,\rho,t_{k}}\Delta t_{k-1}-1}.
\end{eqnarray}
Therefore,
\begin{eqnarray}\label{0.9}
\mathbb{E}\Big[\sup_{k=1,2,...,n_{T}}\vert J_{1}\vert^{2q}\Big]&=&\mathbb{E}\Big[\sup_{k=1,2,...,n_{T}}\big\vert \frac{M_{\gamma,\rho,t_{k}}\Xi_{k}}{M_{\gamma,\rho,t_{k}}\Delta t_{k-1}-1}\big\vert^{2q}\Big]\nonumber\\
&=&\mathbb{E}\Big[\sup_{k=1,2,...,n_{T}}\big\vert \frac{M_{\gamma,\rho,t_{k}}\Delta t_{k-1}}{M_{\gamma,\rho,t_{k}}\Delta t_{k-1}-1}\big\vert^{2q}\big\vert \frac{\Xi_{k}}{\Delta t_{k-1}}\big\vert^{2q}\Big]\nonumber\\
&\leq&\Big(1\vee\big(\frac{\xi}{1-\xi}\big)^{2q}\Big)\mathbb{E}\Big[\sup_{k=1,2,...,n_{T}}\big\vert\frac{\Xi_{k}}{\Delta t_{k-1}}\big\vert^{2q}\Big].
\end{eqnarray}
The last inequality in $\eqref{0.9}$ results from the facts that $\big\vert\frac{z\Delta t}{z\Delta t-1}\big\vert^{2q}$ increase to 1 as $z(z<0)$ decreasing to $-\infty$ for all $\Delta t\in(0,1)$ and $M_{\gamma,\rho,t_{k}}\Delta t_{k-1}\leq Q_{\gamma,\rho}\Delta t<\xi$ a.s. where $\xi\in(0,1)$. According to the proof of Theorem $\ref{thm3.5}$, we have
\begin{eqnarray}
	\Xi_{k}\!\!&=&\!\!\int_{t_{k-1}}^{t_{k}}(s-t_{k-1})\big[(F'_{\gamma,\rho}F_{\gamma,\rho})(Z^{t_{k-1},Z_{t_{k-1}}}_{s-})+\frac{1}{2}(1-\rho)^{2}\alpha_{3}^{2}F''_{\gamma,\rho}(Z^{t_{k-1},Z_{t_{k-1}}}_{s-})\big]\,ds\nonumber\\
	&&+(1-\rho)\alpha_{3}\int_{t_{k-1}}^{t_{k}}(s-t_{k-1})F'_{\gamma,\rho}(Z^{t_{k-1},Z_{t_{k-1}}}_{s-})\,dW_s, \ \  k\in\{1,2,...,n_{T}\}.
\end{eqnarray}
By the H\"{o}lder inequality, we deduce
\begin{eqnarray}
&&\mathbb{E}\Big[\sup_{k=1,2,...,n_{T}}\big\vert\frac{\Xi_{k}}{\Delta t_{k-1}}\big\vert^{2q}\Big]\nonumber\\
\!\!&=&\!\!\mathbb{E}\Big[\sup_{k=1,2,...,n_{T}}\Big\vert\int_{t_{k-1}}^{t_{k}}\frac{s-t_{k-1}}{\Delta t_{k-1}}\big[(F'_{\gamma,\rho}F_{\gamma,\rho})(Z^{t_{k-1},Z_{t_{k-1}}}_{s-})+\frac{1}{2}(1-\rho)^{2}\alpha_{3}^{2}\nonumber\\
&&\cdot F''_{\gamma,\rho}(Z^{t_{k-1},Z_{t_{k-1}}}_{s-})\big]\,ds+\frac{(1-\rho)\alpha_{3}}{\Delta t_{k-1}}\int_{t_{k-1}}^{t_{k}}(s-t_{k-1})F'_{\gamma,\rho}(Z^{t_{k-1},Z_{t_{k-1}}}_{s-})\,dW_s\Big\vert^{2q}\Big]\nonumber\\
\!\!&\leq&\!\!C\mathbb{E}\Big[\sup_{k=1,2,...,n_{T}}\Big\vert\int_{t_{k-1}}^{t_{k}}\frac{s-t_{k-1}}{\Delta t_{k-1}}\big[(F'_{\gamma,\rho}F_{\gamma,\rho})(Z^{t_{k-1},Z_{t_{k-1}}}_{s-})\nonumber\\
&&+\frac{1}{2}(1-\rho)^{2}\alpha_{3}^{2}F''_{\gamma,\rho}(Z^{t_{k-1},Z_{t_{k-1}}}_{s-})\big]\,ds\Big\vert^{2q}\Big]\nonumber\\
&&+C\mathbb{E}\Big[\sup_{k=1,2,...,n_{T}}\Big\vert\int_{t_{k-1}}^{t_{k}}\frac{s-t_{k-1}}{\Delta t_{k-1}}F'_{\gamma,\rho}(Z^{t_{k-1},Z_{t_{k-1}}}_{s-})\,dW_s\Big\vert^{2q}\Big]\nonumber\\
\!\!&\leq&\!\!C\mathbb{E}\Big[\sup_{k=1,2,...,n_{T}}\Big\vert\int_{t_{k-1}}^{t_{k}}\frac{s-t_{\lfloor s\rfloor}}{\Delta t_{\lfloor s\rfloor}}F'_{\gamma,\rho}(Z^{t_{\lfloor s\rfloor},Z_{t_{\lfloor s\rfloor}}}_{s-})\,dW_s\Big\vert^{2q}\Big]+C(\Delta t)^{2q-1}\nonumber\\
&&\cdot\mathbb{E}\Big[\!\sup_{k=1,2,...,n_{T}}\!\int_{t_{k-1}}^{t_{k}}\!\!\big\vert(F'_{\gamma,\rho}F_{\gamma,\rho})(Z^{t_{\lfloor s\rfloor},Z_{t_{\lfloor s\rfloor}}}_{s-})+\frac{1}{2}(1-\rho)^{2}\alpha_{3}^{2}F''_{\gamma,\rho}(Z^{t_{\lfloor s\rfloor},Z_{t_{\lfloor s\rfloor}}}_{s-})\big\vert^{2q}\!\,ds\Big]\nonumber.
\end{eqnarray}
Then $\eqref{3.9}$ and the BDG inequality lead to
\begin{eqnarray}\label{3.40}
&&\mathbb{E}\Big[\sup_{k=1,2,...,n_{T}}\big\vert\frac{\Xi_{k}}{\Delta t_{k-1}}\big\vert^{2q}\Big]\nonumber\\
\!\!&\leq&\!\!C(\Delta t)^{2q}\mathbb{E}\Big[\sup_{s\in[0,T]}\big\vert(F'_{\gamma,\rho}F_{\gamma,\rho})(Z^{t_{\lfloor s\rfloor},Z_{t_{\lfloor s\rfloor}}}_{s-})+\frac{1}{2}(1-\rho)^{2}\alpha_{3}^{2}F''_{\gamma,\rho}(Z^{t_{\lfloor s\rfloor},Z_{t_{\lfloor s\rfloor}}}_{s-})\big\vert^{2q}\Big]\nonumber\\
&&+C\mathbb{E}\Big[\sup_{k=1,2,...,n_{T}}\Big\vert\int_{0}^{t_{k}}\frac{s-t_{\lfloor s\rfloor}}{\Delta t_{\lfloor s\rfloor}}F'_{\gamma,\rho}(Z^{t_{\lfloor s\rfloor},Z_{t_{\lfloor s\rfloor}}}_{s-})\,dW_s\Big\vert^{2q}\Big]\nonumber\\
\!\!&\leq&\!\!C(\Delta t)^{2q}+C\mathbb{E}\Big[\sup_{t\in[0,T]}\Big\vert\int_{0}^{t}\frac{s-t_{\lfloor s\rfloor}}{\Delta t_{\lfloor s\rfloor}}F'_{\gamma,\rho}(Z^{t_{\lfloor s\rfloor},Z_{t_{\lfloor s\rfloor}}}_{s-})\,dW_s\Big\vert^{2q}\Big]\nonumber\\
\!\!&\leq&\!\!C(\Delta t)^{2q}+C\mathbb{E}\Big[\Big(\int_{0}^{T}\big\vert\frac{s-t_{\lfloor s\rfloor}}{\Delta t_{\lfloor s\rfloor}}F'_{\gamma,\rho}(Z^{t_{\lfloor s\rfloor},Z_{t_{\lfloor s\rfloor}}}_{s-})\big\vert^{2}\,ds\Big)^{q}\Big]\nonumber\\
\!\!&\leq&\!\! C.
\end{eqnarray}
Combining $\eqref{0.9}$ and \eqref{3.40} gives
\begin{eqnarray}
\mathbb{E}\big[\sup_{k=1,2,...,n_{T}}\vert J_{1}\vert^{2q}\big]\leq C.	
\end{eqnarray}
To estimate $J_{2}$, we note that
\begin{eqnarray}
	\!\!&&F_{\gamma,\rho}\big(\Phi(Z_{t_{k-1}},t_{k-1},\Delta t_{k-1})\big)-F_{\gamma,\rho}\big(\Phi(\bar{Z}_{t_{k-1}},t_{k-1},\Delta t_{k-1})\big)\nonumber\\
	\!\!&=&\!\!\big(\Phi(Z_{t_{k-1}},t_{k-1},\Delta t_{k-1})-\Phi(\bar{Z}_{t_{k-1}},t_{k-1},\Delta t_{k-1})\big)\!\!\int_{0}^{1}F'_{\gamma,\rho}\big(\Phi(Z_{t_{k-1}},t_{k-1},\Delta t_{k-1})\nonumber\\
	&&+\theta(\Phi(\bar{Z}_{t_{k-1}},t_{k-1},\Delta t_{k-1})\!-\!\Phi(Z_{t_{k-1}},t_{k-1},\Delta t_{k-1}))\big)\,d\theta\nonumber.
\end{eqnarray}
Let
\begin{eqnarray}
	\bar{M}_{\gamma,\rho,t_{k}}\!\!&:=&\!\!\int_{0}^{1}F'_{\gamma,\rho}\big(\Phi(Z_{t_{k-1}},t_{k-1},\Delta t_{k-1})\nonumber\\
	&&+\theta(\Phi(\bar{Z}_{t_{k-1}},t_{k-1},\Delta t_{k-1})\!-\!\Phi(Z_{t_{k-1}},t_{k-1},\Delta t_{k-1}))\big)\,d\theta.
\end{eqnarray}
By $\eqref{3.6}$, we have $\bar{M}_{\gamma,\rho,t_{k}}\leq Q_{\gamma,\rho}$, a.s. In light of $\eqref{0.2}$, we deduce
\begin{eqnarray}
	&&\Phi(Z_{t_{k-1}},t_{k-1},\Delta t_{k-1})-\Phi(\bar{Z}_{t_{k-1}},t_{k-1},\Delta t_{k-1})\nonumber\\
	&=&Z_{t_{k-1}}\!-\!\bar{Z}_{t_{k-1}}+\bar{M}_{\gamma,\rho,t_{k}}\Delta t_{k-1}[\Phi(Z_{t_{k-1}},t_{k-1},\Delta t_{k-1})\!-\!\Phi(\bar{Z}_{t_{k-1}},t_{k-1},\Delta t_{k-1})]\nonumber
\end{eqnarray}
and 
\begin{eqnarray}
	\Phi(Z_{t_{k-1}},t_{k-1},\Delta t_{k-1})-\Phi(\bar{Z}_{t_{k-1}},t_{k-1},\Delta t_{k-1})=\frac{Z_{t_{k-1}}\!-\!\bar{Z}_{t_{k-1}}}{1-\bar{M}_{\gamma,\rho,t_{k}}\Delta t_{k-1}}\nonumber.
\end{eqnarray}
According to Theorem $\ref{thm3.5}$, we thus infer
\begin{eqnarray}
	\mathbb{E}\Big[\sup_{k=1,2,...,n_{T}}\vert \Delta t_{k-1}J_{2}\vert^{2q}\Big]\!&=&\!\mathbb{E}\Big[\sup_{k=1,2,...,n_{T}}\Big\vert\frac{\bar{M}_{\gamma,\rho,t_{k}}\Delta t_{k-1}}{1-\bar{M}_{\gamma,\rho,t_{k}}\Delta t_{k-1}}(Z_{t_{k-1}}\!-\!\bar{Z}_{t_{k-1}})\Big\vert^{2q}\Big]\nonumber\\
	\!&\leq&\!\Big(1\vee\big(\frac{\xi}{1-\xi}\big)^{2q}\Big)\mathbb{E}\Big[\sup_{k=1,2,...,n_{T}}\big\vert Z_{t_{k-1}}\!-\!\bar{Z}_{t_{k-1}}\big\vert^{2q}\Big]\nonumber\\
	\!&\leq&\!C(\Delta t)^{2q}.
\end{eqnarray}
Consequently,
\begin{eqnarray}\label{81}
	\mathbb{E}\Big[\sup_{k=1,2,...,n_{T}}\big\vert\Delta t_{k-1}(F_{\gamma,\rho}(Z_{t_{k}-})-F_{\gamma,\rho}(\bar{Z}_{t_{k}-}))\big\vert^{2q}\Big]\leq C(\Delta t)^{2q}.
\end{eqnarray}
Using the triangle inequality and \eqref{81} yields
\begin{eqnarray}
&&\mathbb{E}\Big[\sup_{k=1,2,...,n_{T}}\big\vert \Delta t_{k-1}F_{\gamma,\rho}(\bar{Z}_{t_{k}-})\big\vert^{2q}\Big]\nonumber\\
&\leq&C\mathbb{E}\Big[\sup_{k=1,2,...,n_{T}}\big\vert \Delta t_{k-1}F_{\gamma,\rho}(Z_{t_{k}-})\big\vert^{2q}\Big]\nonumber\\
&&+C\mathbb{E}\Big[\sup_{k=1,2,...,n_{T}}\big\vert\Delta t_{k-1}(F_{\gamma,\rho}(Z_{t_{k}-})-F_{\gamma,\rho}(\bar{Z}_{t_{k}-}))\big\vert^{2q}\Big]\nonumber\\
&\leq& C(\Delta t)^{2q}+C\mathbb{E}\Big[\sup_{k=1,2,...,n_{T}}\!\!\big\vert\Delta t_{k-1} F_{\gamma,\rho}(Z_{t_{k}-})\big\vert^{2q}\Big].\nonumber
\end{eqnarray}
The proof is completed.	
\end{proof}

\subsection*{Proof of Lemma \ref{lem3.9}.}
\begin{proof}
Let $\gamma> 2\rho-1$. For any $q\geq 1$, $\varepsilon\in(0,1)$ and $k\in\{1,2,...,n_{T}\}$, using the Cauchy--Bunyakovsky--Schwarz inequality gives 
\begin{eqnarray}\label{65}
\!\!&&\!\!\mathbb{E}^{N}\Big[\bar{Z}_{t_{k}-}^{-2mq}\mathbf{1}_{\{\bar{Z}_{t_{k}-}\leq(\Delta t)^{\varepsilon}\bar{Z}_{t_{k-1}}\}}\Big]\nonumber\\
\!\!&=&\!\!\mathbb{E}^{N}\Big[\bar{Z}_{t_{k}-}^{-2mq}(\Delta t_{k-1})^{2q}\cdot\frac{1}{(\Delta t_{k-1})^{2q}}\mathbf{1}_{\{\bar{Z}_{t_{k}-}\leq(\Delta t)^{\varepsilon}\bar{Z}_{t_{k-1}}\}}\Big]\nonumber\\
\!\!&\leq&\!\!(\Delta t)^{2q}\!\Big(\mathbb{E}^{N}\!\Big[\bar{Z}_{t_{k}-}^{-4mq}\big(\frac{\Delta t_{k\!-\!1}}{\Delta t}\big)^{4q}\Big]\Big)^{\frac{1}{2}}\!\Big(\mathbb{E}^{N}\!\Big[\frac{1}{(\Delta t_{k\!-\!1})^{4q}}\mathbf{1}_{\{\bar{Z}_{t_{k}-}\leq(\Delta t)^{\varepsilon}\bar{Z}_{t_{k\!-\!1}}\}}\Big]\Big)^{\frac{1}{2}}.
\end{eqnarray}	
By \eqref{3.6}, for any $x>y>0$, we have
\begin{eqnarray}
	F_{\gamma,\rho}(y)\geq F_{\gamma,\rho}(x)-Q_{\gamma,\rho}(x-y).
\end{eqnarray}
It follows from \eqref{11} and \eqref{3.4} that if $\bar{Z}_{t_{k}-}\leq(\Delta t)^{\varepsilon}\bar{Z}_{t_{k-1}}$, 
\begin{eqnarray}\label{a86}
	\!\!\Delta W_{k-1}\!\!&\geq&\!\! \frac{(1\!-\!(\Delta t)^{\varepsilon})\bar{Z}_{t_{k-1}}\!+\![F_{\gamma,\rho}((\Delta t)^{\varepsilon}\bar{Z}_{t_{k-1}})\!-\!Q_{\gamma,\rho}((\Delta t)^{\varepsilon}\bar{Z}_{t_{k-1}}\!-\!\bar{Z}_{t_{k}-})]\Delta t_{k-1}}{(\rho-1)\alpha_{3}}\nonumber\\
\!&>&\! \frac{(1\!-\!(\Delta t)^{\varepsilon})\bar{Z}_{t_{k-1}}\!\!+\!(\rho-1)\big[\alpha_{2}(\Delta t)^{-m\varepsilon}\bar{Z}_{t_{k-1}}^{-m}+\frac{\rho\alpha_{3}^{2}}{2}(\Delta t)^{-\varepsilon}\bar{Z}_{t_{k-1}}^{-1}\big]\Delta t_{k-1}}{(\rho-1)\alpha_{3}}\nonumber\\
\!&&\!-\frac{\big[\alpha_{-1}(\Delta t)^{\frac{\rho+1}{\rho-1}\varepsilon}\bar{Z}_{t_{k-1}}^{\frac{\rho+1}{\rho-1}}\!+\!(\alpha_{1}+Q_{\gamma,\rho}/(\rho-1))(\Delta t)^{\varepsilon}\bar{Z}_{t_{k-1}}\big]\Delta t_{k-1}}{\alpha_{3}}.
\end{eqnarray}
Let 
\begin{tiny}
\begin{eqnarray}
	\!\!A_{(k-1)}\!\!:=\!\!\Bigg\{\omega\colon \Delta W_{k-1}\!\!&>&\!\! \frac{(1\!-\!(\Delta t)^{\varepsilon})\bar{Z}_{t_{k-1}}\!\!+\!(\rho-1)\Big[\alpha_{2}(\Delta t)^{-m\varepsilon}\bar{Z}_{t_{k-1}}^{-m}+\frac{\rho\alpha_{3}^{2}}{2}(\Delta t)^{-\varepsilon}\bar{Z}_{t_{k-1}}^{-1}\Big]\Delta t_{k-1}}{(\rho-1)\alpha_{3}}\nonumber\\
\!\!&&\!\!-\frac{\Big[\alpha_{-1}(\Delta t)^{\frac{\rho+1}{\rho-1}\varepsilon}\bar{Z}_{t_{k-1}}^{\frac{\rho+1}{\rho-1}}\!+\!(\alpha_{1}+Q_{\gamma,\rho}/(\rho-1))(\Delta t)^{\varepsilon}\bar{Z}_{t_{k-1}}\Big]\Delta t_{k-1}}{\alpha_{3}},\nonumber\\
\!\!&&\!\! (\Delta t_{k-1})^{\varepsilon}\bar{Z}_{t_{k-1}}^{\frac{2}{\rho-1}}\!\!\leq\! 1\Bigg\}\nonumber.
\end{eqnarray}
\end{tiny}
One can infer that if $(\Delta t_{k-1})^{\varepsilon}\bar{Z}_{t_{k-1}}^{\frac{2}{\rho-1}}\leq 1$,
\begin{small}
\begin{eqnarray}\label{a71}
\!\!&&\!\!\frac{(1\!-\!(\Delta t)^{\varepsilon})\bar{Z}_{t_{k-1}}\!\!+\!(\rho-1)\Big[\alpha_{2}(\Delta t)^{-m\varepsilon}\bar{Z}_{t_{k-1}}^{-m}+\frac{\rho\alpha_{3}^{2}}{2}(\Delta t)^{-\varepsilon}\bar{Z}_{t_{k-1}}^{-1}\Big]\Delta t_{k-1}}{(\rho-1)\alpha_{3}}\nonumber\\
\!\!&&\!\!-\frac{\Big[\alpha_{-1}(\Delta t)^{\frac{\rho+1}{\rho-1}\varepsilon}\bar{Z}_{t_{k-1}}^{\frac{\rho+1}{\rho-1}}\!+\!(\alpha_{1}+Q_{\gamma,\rho}/(\rho-1))(\Delta t)^{\varepsilon}\bar{Z}_{t_{k-1}}\Big]\Delta t_{k-1}}{\alpha_{3}}\nonumber\\
\!\!&>&\!\!\frac{\big[1\!-\!\big((\Delta t)^{\varepsilon}\!+\!(\rho\!-\!1)\alpha_{-1}(\Delta t)^{\frac{\rho\!+\!1}{\rho\!-\!1}\varepsilon}(\Delta t_{k\!-\!1})^{1\!-\!\varepsilon}\!+\!((\rho\!-\!1)\alpha_{1}\!+\!Q_{\gamma,\rho})(\Delta t)^{\varepsilon}\Delta t_{k\!-\!1}\big)\big]\bar{Z}_{t_{k\!-\!1}}}{(\rho\!-\!1)\alpha_{3}}\nonumber\\
\!\!&&\!\!+\alpha_{3}^{-1}\alpha_{2}(\Delta t)^{-m\varepsilon}\bar{Z}_{t_{k-1}}^{-m}\Delta t_{k-1}.
\end{eqnarray}
\end{small}
With sufficiently small $\Delta t$ satisfying
\begin{small}
\begin{eqnarray}\label{69}
	\!\!(\Delta t)^{\varepsilon}\!<\!\frac{1}{2(1\!+\!(\rho\!-\!1)(\alpha_{-1}\!+\!\alpha_{1})\!+\!Q_{\gamma,\rho})},\ \  \!(\Delta t)^{\frac{m\!-\!1}{2m}\!+\varepsilon}\!\leq\! \frac{\alpha_{2}^{1/m}}{2(\rho\!-\!1)\alpha_{3}^{(m+1)/m}},
\end{eqnarray}
\end{small}
in which the first inequality implies that
\begin{eqnarray}
	\!\!&&\!\!(\Delta t)^{\varepsilon}\!+\!(\rho\!-\!1)\alpha_{-1}(\Delta t)^{\frac{\rho\!+\!1}{\rho\!-\!1}\varepsilon}(\Delta t_{k\!-\!1})^{1\!-\!\varepsilon}\!+\!((\rho\!-\!1)\alpha_{1}\!+\!Q_{\gamma,\rho})(\Delta t)^{\varepsilon}\Delta t_{k\!-\!1}\nonumber\\
	\!\!&\leq&\!\!(\Delta t)^{\varepsilon}+(\rho-1)\alpha_{-1}(\Delta t)^{1+\frac{2}{\rho\!-\!1}\varepsilon}+((\rho\!-\!1)\alpha_{1}\!+\!Q_{\gamma,\rho})(\Delta t)^{1+\varepsilon}\nonumber\\
	\!\!&\leq&\!\!\big(1+(\rho-1)(\alpha_{-1}+\alpha_{1})+Q_{\gamma,\rho}\big)(\Delta t)^{\varepsilon}<\frac{1}{2}
\end{eqnarray}
holds almost surely, we divide the whole space $\Omega$ into two parts $\big\{\bar{Z}_{t_{k-1}}\!\geq\! 2(\rho-1)\alpha_{3}\sqrt{\Delta t_{k-1}}\big\}$ and $\big\{\bar{Z}_{t_{k-1}}\!<\! 2(\rho-1)\alpha_{3}\sqrt{\Delta t_{k-1}}\big\}$, and further estimate \eqref{a71} gives
\begin{eqnarray}\label{a74}
\!\!&&\!\!\frac{(1\!-\!(\Delta t)^{\varepsilon})\bar{Z}_{t_{k-1}}\!\!+\!(\rho-1)\Big[\alpha_{2}(\Delta t)^{-m\varepsilon}\bar{Z}_{t_{k-1}}^{-m}+\frac{\rho\alpha_{3}^{2}}{2}(\Delta t)^{-\varepsilon}\bar{Z}_{t_{k-1}}^{-1}\Big]\Delta t_{k-1}}{(\rho-1)\alpha_{3}}\nonumber\\
\!\!&&\!\!-\frac{\Big[\alpha_{-1}(\Delta t)^{\frac{\rho+1}{\rho-1}\varepsilon}\bar{Z}_{t_{k-1}}^{\frac{\rho+1}{\rho-1}}\!+\!(\alpha_{1}+Q_{\gamma,\rho}/(\rho-1))(\Delta t)^{\varepsilon}\bar{Z}_{t_{k-1}}\Big]\Delta t_{k-1}}{\alpha_{3}}\nonumber\\
\!\!&>&\!\!\frac{\bar{Z}_{t_{k-1}}}{2(\rho-1)\alpha_{3}}+\frac{\alpha_{2}\sqrt{\Delta t_{k-1}}}{\alpha_{3}(\Delta t)^{m\varepsilon}\bar{Z}_{t_{k-1}}^{m}}\sqrt{\Delta t_{k-1}}\nonumber\\
\!\!&>&\!\!\sqrt{\Delta t_{k-1}}\cdot\mathbf{1}_{\{\bar{Z}_{t_{k-1}}\geq 2(\rho-1)\alpha_{3}\sqrt{\Delta t_{k-1}}\}}\nonumber\\
\!\!&&\!\!+\frac{\alpha_{2}}{\alpha_{3}(\Delta t)^{m\varepsilon}}(2(\rho-1)\alpha_{3})^{-m}(\Delta  t_{k-1})^{\frac{1-m}{2}}\sqrt{\Delta t_{k-1}}\cdot\mathbf{1}_{\{\bar{Z}_{t_{k-1}}<2(\rho-1)\alpha_{3}\sqrt{\Delta t_{k-1}}\}}\nonumber\\
\!\!&>&\!\! \sqrt{\Delta t_{k-1}},
\end{eqnarray}
where we have used the second inequality of \eqref{69} in the last inequality. It is well known that for a standard normal random variable $G$ and constant $\beta\!>\!0$, there is a standard inequality for the lower tail of the normal distribution, namely
 \begin{eqnarray}\label{90}
 	\mathbb{P}(G\geq\beta)\leq \frac{1}{\sqrt{2\pi}\beta}\exp(-\frac{\beta^{2}}{2}).
 \end{eqnarray}
Then from the definition of Wiener process, it follows that for time nodes $t_{k-1}, t_{k}\in\mathbb{R}, 0\leq t_{k-1}<t_{k}\leq T$ and constant $a>0$, 
\begin{eqnarray}
	\mathbb{P}(\Delta W_{k-1}\geq a)\leq \frac{\sqrt{\Delta t_{k-1}}}{\sqrt{2\pi}a}\exp(-\frac{a^{2}}{2\Delta t_{k-1}})
\end{eqnarray}
according to $\Delta W_{k-1}\sim N(0,\Delta t_{k-1})$. Conditioning on $\sigma(\mathcal{F}_{T}^{N}\cup\mathcal{F}_{t_{k-1}}^{W})$, we find that for the Wiener increment appeared in \eqref{a86}, 
\begin{tiny}
\begin{eqnarray}
	\mathbb{P}\Bigg(\Delta W_{k-1}\!\!&>&\!\! \frac{(1\!-\!(\Delta t)^{\varepsilon})\bar{Z}_{t_{k-1}}\!\!+\!(\rho-1)\big[\alpha_{2}(\Delta t)^{-m\varepsilon}\bar{Z}_{t_{k-1}}^{-m}+\frac{\rho\alpha_{3}^{2}}{2}(\Delta t)^{-\varepsilon}\bar{Z}_{t_{k-1}}^{-1}\big]\Delta t_{k-1}}{(\rho-1)\alpha_{3}}\nonumber\\
	\!\!&&\!\!-\frac{\big[\alpha_{-1}(\Delta t)^{\frac{\rho+1}{\rho-1}\varepsilon}\bar{Z}_{t_{k-1}}^{\frac{\rho+1}{\rho-1}}\!+\!(\alpha_{1}+Q_{\gamma,\rho}/(\rho-1))(\Delta t)^{\varepsilon}\bar{Z}_{t_{k-1}}\big]\Delta t_{k-1}}{\alpha_{3}}\Bigg\vert\sigma(\mathcal{F}_{T}^{N}\cup\mathcal{F}_{t_{k-1}}^{W})\Bigg)\nonumber\\
\!\!&\leq&\!\!	\exp\Big\{-\Big((1\!-\!(\Delta t)^{\varepsilon})\bar{Z}_{t_{k\!-\!1}}+(\rho\!-\!1)\Delta t_{k-1}\Big[\alpha_{2}(\Delta t)^{-m\varepsilon}\bar{Z}_{t_{k-1}}^{-m}\nonumber\\
\!\!&&\!\!+\frac{\rho\alpha_{3}^{2}}{2}(\Delta t)^{-\varepsilon}\bar{Z}_{t_{k-1}}^{-1}\!-\!\big(\alpha_{-1}(\Delta t)^{\frac{\rho\!+\!1}{\rho\!-\!1}\varepsilon}\bar{Z}_{t_{k\!-\!1}}^{\frac{\rho\!+\!1}{\rho\!-\!1}}\!+\!(\alpha_{1}\!+\!Q_{\gamma,\rho}/(\rho\!-\!1))(\Delta t)^{\varepsilon}\bar{Z}_{t_{k\!-\!1}}\big)\Big]\Big)^{2}\nonumber\\
\!\!&&\!\!\cdot\frac{1}{2(\rho-1)^{2}\alpha_{3}^{2}\Delta t_{k-1}}\Big\}\nonumber
\end{eqnarray}
\end{tiny}
when $(\Delta t_{k-1})^{\varepsilon}\bar{Z}_{t_{k-1}}^{\frac{2}{\rho-1}}\leq 1$, where we note that \eqref{a74} implies that the factor ``$\frac{\sqrt{\Delta t_{k-1}}}{\sqrt{2\pi}a}$'' is bounded by 1 when
\begin{eqnarray}
	a\!\!&:=&\!\! \frac{(1\!-\!(\Delta t)^{\varepsilon})\bar{Z}_{t_{k-1}}\!\!+\!(\rho-1)\big[\alpha_{2}(\Delta t)^{-m\varepsilon}\bar{Z}_{t_{k-1}}^{-m}+\frac{\rho\alpha_{3}^{2}}{2}(\Delta t)^{-\varepsilon}\bar{Z}_{t_{k-1}}^{-1}\big]\Delta t_{k-1}}{(\rho-1)\alpha_{3}}\nonumber\\
	\!\!&&\!\!-\frac{\big[\alpha_{-1}(\Delta t)^{\frac{\rho+1}{\rho-1}\varepsilon}\bar{Z}_{t_{k-1}}^{\frac{\rho+1}{\rho-1}}\!+\!(\alpha_{1}+Q_{\gamma,\rho}/(\rho-1))(\Delta t)^{\varepsilon}\bar{Z}_{t_{k-1}}\big]\Delta t_{k-1}}{\alpha_{3}}.\nonumber
\end{eqnarray}
Combining \eqref{a86}, the properties of conditional expectation and the result above, we deduce
\begin{eqnarray}\label{68}
\!\!&&\!\!\mathbb{E}^{N}\Big[\frac{1}{(\Delta t_{k-1})^{4q}}\mathbf{1}_{\{\bar{Z}_{t_{k}-}\leq(\Delta t)^{\varepsilon}\bar{Z}_{t_{k-1}}\}}\Big]\nonumber\\
\!\!&=&\!\!\mathbb{E}^{N}\Big[\frac{1}{(\Delta t_{k-1})^{4q}}\mathbf{1}_{\{\bar{Z}_{t_{k}-}\leq(\Delta t)^{\varepsilon}\bar{Z}_{t_{k-1}},(\Delta t_{k-1})^{\varepsilon}\bar{Z}_{t_{k-1}}^{\frac{2}{\rho\!-\!1}}\leq1\}}\Big]\nonumber\\
\!\!&&\!\!+\mathbb{E}^{N}\Big[\frac{1}{(\Delta t_{k-1})^{4q}}\mathbf{1}_{\{\bar{Z}_{t_{k}-}\leq(\Delta t)^{\varepsilon}\bar{Z}_{t_{k-1}},(\Delta t_{k-1})^{\varepsilon}\bar{Z}_{t_{k-1}}^{\frac{2}{\rho\!-\!1}}>1\}}\Big]\nonumber\\
\!\!&\leq&\!\!\mathbb{E}^{N}\Big[\frac{1}{(\Delta t_{k-1})^{4q}}\mathbb{E}\big[\mathbf{1}_{A_{(k-1)}}\big\vert\sigma(\mathcal{F}_{T}^{N}\cup\mathcal{F}_{t_{k-1}}^{W})\big]\Big]\nonumber\\
\!\!&&\!\!+\mathbb{E}^{N}\Big[\frac{1}{(\Delta t_{k-1})^{4q}}\mathbf{1}_{\{\bar{Z}_{t_{k}-}\leq(\Delta t)^{\varepsilon}\bar{Z}_{t_{k-1}},(\Delta t_{k-1})^{\varepsilon}\bar{Z}_{t_{k-1}}^{\frac{2}{\rho\!-\!1}}>1\}}\Big]\nonumber\\
\!\!&\leq&\!\!\mathbb{E}^{N}\Big[\frac{1}{(\Delta t_{k-1})^{4q}}
\exp\Big\{-\Big((1\!-\!(\Delta t)^{\varepsilon})\bar{Z}_{t_{k\!-\!1}}+(\rho\!-\!1)\Delta t_{k-1}\Big[\alpha_{2}(\Delta t)^{-m\varepsilon}\bar{Z}_{t_{k-1}}^{-m}\nonumber\\
\!\!&&\!\!+\frac{\rho\alpha_{3}^{2}}{2}(\Delta t)^{-\varepsilon}\bar{Z}_{t_{k-1}}^{-1}\!-\!\big(\alpha_{-1}(\Delta t)^{\frac{\rho\!+\!1}{\rho\!-\!1}\varepsilon}\bar{Z}_{t_{k\!-\!1}}^{\frac{\rho\!+\!1}{\rho\!-\!1}}\!+\!(\alpha_{1}\!+\!Q_{\gamma,\rho}/(\rho\!-\!1))(\Delta t)^{\varepsilon}\bar{Z}_{t_{k\!-\!1}}\big)\Big]\Big)^{2}\nonumber\\
\!\!&&\!\!\cdot\frac{1}{2(\rho\!-\!1)^{2}\alpha_{3}^{2}\Delta t_{k-1}}\Big\}\mathbf{1}_{\{(\Delta t_{k-1})^{\varepsilon}\bar{Z}_{t_{k-1}}^{\frac{2}{\rho\!-\!1}}\leq 1\}}\Big]\!+\!\mathbb{E}^{N}\Big[\frac{1}{(\Delta t_{k\!-\!1})^{4q}}\mathbf{1}_{\{(\Delta t_{k\!-\!1})^{\varepsilon}\bar{Z}_{t_{k\!-\!1}}^{\frac{2}{\rho-1}}>1\}}\Big]\nonumber\\
\!\!&\leq&\!\! \mathbb{E}^{N}\Big[\bar{Z}_{t_{k-1}}^{\frac{8q}{(\rho-1)\varepsilon}}\Big]+\mathbb{E}^{N}\Big[\frac{1}{(\Delta t_{k-1})^{4q}}\exp\Big(-\frac{(1\!-\!(\Delta t)^{\varepsilon})^{2}\bar{Z}_{t_{k-1}}^{2}}{2(\rho-1)^{2}\alpha_{3}^{2}\Delta t_{k-1}}\Big)\nonumber\\
\!\!&&\!\!\cdot \exp\Big(-\frac{(1\!-\!(\Delta t)^{\varepsilon})\big[\alpha_{2}(\Delta t)^{-m\varepsilon}\bar{Z}_{t_{k-1}}^{1-m}+\frac{\rho\alpha_{3}^{2}}{2}(\Delta t)^{-\varepsilon}\big]}{(\rho-1)\alpha_{3}^{2}}\Big)\\
\!\!&&\!\!\cdot\exp\Big(\frac{(1\!-\!(\Delta t)^{\varepsilon})\big[\alpha_{-1}(\Delta t)^{\frac{\rho\!+\!1}{\rho\!-\!1}\varepsilon}(\Delta t_{k\!-\!1})^{-\rho\varepsilon}\!+\!(\alpha_{1}\!+\!\frac{Q_{\gamma,\rho}}{\rho-1})(\Delta t)^{\varepsilon}(\Delta t_{k\!-\!1})^{-(\rho\!-\!1)\varepsilon}\big]}{(\rho\!-\!1)\alpha_{3}^{2}}\Big)\Big].\nonumber
\end{eqnarray}
Now we divide the whole space $\Omega$ into two parts $\big\{\bar{Z}_{t_{k-1}}\geq (\Delta t_{k-1})^{\frac{1}{3}-\frac{\rho\varepsilon}{2}}\big\}$ and $\big\{\bar{Z}_{t_{k-1}}< (\Delta t_{k-1})^{\frac{1}{3}-\frac{\rho\varepsilon}{2}}\big\}$. The estimate \eqref{68} implies
\begin{eqnarray}
\!\!&&\!\!\mathbb{E}^{N}\Big[\frac{1}{(\Delta t_{k-1})^{4q}}\mathbf{1}_{\{\bar{Z}_{t_{k}-}\leq(\Delta t)^{\varepsilon}\bar{Z}_{t_{k-1}}\}}\Big]\nonumber\\
\!\!&\leq&\!\!\mathbb{E}^{N}\Big[\bar{Z}_{t_{k-1}}^{\frac{8q}{(\rho-1)\varepsilon}}\Big]+\mathbb{E}^{N}\Big[\frac{1}{(\Delta t_{k-1})^{4q}}\exp\Big(-\frac{(1\!-\!(\Delta t)^{\varepsilon})^{2}}{2(\rho-1)^{2}\alpha_{3}^{2}(\Delta t_{k-1})^{\frac{1}{3}+\rho\varepsilon}}\Big)\nonumber\\
\!\!&&\!\!\cdot\exp\Big(\frac{(1\!-\!(\Delta t)^{\varepsilon})\big(\alpha_{-1}\!+\!\alpha_{1}+Q_{\gamma,\rho}/(\rho\!-\!1)\big)(\Delta t)^{\varepsilon}}{(\rho-1)\alpha_{3}^{2}(\Delta t_{k-1})^{\rho\varepsilon}}\Big)\Big]\nonumber\\
\!\!&&\!\!+\mathbb{E}^{N}\Big[\frac{1}{(\Delta t_{k-1})^{4q}} \exp\Big(-\frac{\alpha_{2}(1\!-\!(\Delta t)^{\varepsilon})}{(\rho-1)\alpha_{3}^{2}(\Delta t)^{m\varepsilon}(\Delta t_{k-1})^{(\frac{1}{3}-\frac{\rho\varepsilon}{2})(m-1)}}\Big)\nonumber\\
\!\!&&\!\!\cdot\exp\Big(\frac{(1\!-\!(\Delta t)^{\varepsilon})\big(\alpha_{-1}\!+\!\alpha_{1}+Q_{\gamma,\rho}/(\rho\!-\!1)\big)(\Delta t)^{\varepsilon}}{(\rho-1)\alpha_{3}^{2}(\Delta t_{k-1})^{\rho\varepsilon}}\Big)\Big]\nonumber.
\end{eqnarray}
For $\gamma>2\rho-1$, if $\varepsilon<\frac{2(\gamma+1-2\rho)}{3\rho(\gamma-1)}$, we have $\frac{m-1}{3}-\frac{\rho(m+1)\varepsilon}{2}>0$, where $m=\frac{\gamma-\rho}{\rho-1}$. With sufficiently small $\Delta t$ satisfying \eqref{69} and
\begin{eqnarray}
	\!\!(\Delta t)^{\varepsilon}\!\!&<&\!\!\Big(\frac{\alpha_{2}(\rho\!-\!1)}{2(\rho\!-\!1)(\alpha_{-1}\!+\!\alpha_{1})\!+\!2Q_{\gamma,\rho})}\Big)^{\frac{1}{m+1}}\wedge
	\frac{1}{2\!+\!4(\rho\!-\!1)(\alpha_{-1}\!+\!\alpha_{1})\!+\!4Q_{\gamma,\rho}},\nonumber
\end{eqnarray}
which implies that the following inequalities
\begin{eqnarray}
1-	(\Delta t)^{\varepsilon}-2(\rho-1)(\alpha_{-1}+\alpha_{1}+Q_{\gamma,\rho}/(\rho\!-\!1))(\Delta t)^{\varepsilon}(\Delta t_{k-1})^{1/3}>1/2,\nonumber\\
\alpha_{2}-(\alpha_{-1}+\alpha_{1}+Q_{\gamma,\rho}/(\rho\!-\!1))(\Delta t)^{(m+1)\varepsilon}(\Delta t_{k-1})^{\frac{m-1}{3}-\frac{\rho(m+1)\varepsilon}{2}}>\frac{\alpha_{2}}{2}
\end{eqnarray}
hold almost surely, we deduce
\begin{eqnarray}\label{71}
\!\!&&\!\!\mathbb{E}^{N}\Big[\frac{1}{(\Delta t_{k-1})^{4q}}\mathbf{1}_{\{\bar{Z}_{t_{k}-}\leq(\Delta t)^{\varepsilon}\bar{Z}_{t_{k-1}}\}}\Big]\nonumber\\
\!\!&\leq&\!\!\mathbb{E}^{N}\Big[\bar{Z}_{t_{k-1}}^{\frac{8q}{(\rho-1)\varepsilon}}\Big]\!+\!\mathbb{E}^{N}\Big[\frac{1}{(\Delta t_{k\!-\!1})^{4q}}\exp\Big(\!\!-\!\frac{(1\!-\!(\Delta t)^{\varepsilon})\big(\alpha_{-1}\!+\!\alpha_{1}\!+\!Q_{\gamma,\rho}/(\rho\!-\!1)\big)(\Delta t)^{\varepsilon}}{(\rho\!-\!1)\alpha_{3}^{2}(\Delta t_{k-1})^{\rho\varepsilon}}\nonumber\\
\!\!&&\!\!\cdot\Big(\frac{1-(\Delta t)^{\varepsilon}}{2(\rho-1)(\alpha_{-1}+\alpha_{1}+Q_{\gamma,\rho}/(\rho\!-\!1))(\Delta t)^{\varepsilon}(\Delta t_{k-1})^{1/3}}-1\Big)\Big)\Big]\nonumber\\
\!\!&&\!\!+\mathbb{E}^{N}\Big[\frac{1}{(\Delta t_{k-1})^{4q}}\exp\Big(-\frac{(1\!-\!(\Delta t)^{\varepsilon})\big(\alpha_{-1}\!+\!\alpha_{1}+Q_{\gamma,\rho}/(\rho\!-\!1)\big)(\Delta t)^{\varepsilon}}{(\rho-1)\alpha_{3}^{2}(\Delta t_{k-1})^{\rho\varepsilon}}\nonumber\\
\!\!&&\!\!\cdot\Big(\frac{\alpha_{2}}{(\alpha_{-1}+\alpha_{1}+Q_{\gamma,\rho}/(\rho\!-\!1))(\Delta t)^{(m+1)\varepsilon}(\Delta t_{k-1})^{\frac{m-1}{3}-\frac{\rho(m+1)\varepsilon}{2}}}-1\Big)\Big)\Big]\nonumber\\
\!\!&\leq&\!\!\mathbb{E}^{N}\Big[\bar{Z}_{t_{k-1}}^{\frac{8q}{(\rho-1)\varepsilon}}\Big]+\mathbb{E}^{N}\Big[\frac{1}{(\Delta t_{k-1})^{4q}}\exp\Big(-\frac{1}{8(\rho-1)^{2}\alpha_{3}^{2}(\Delta t_{k-1})^{\frac{1}{3}+\rho\varepsilon}}\Big)\Big]\nonumber\\
\!\!&&\!\!+\mathbb{E}^{N}\Big[\frac{1}{(\Delta t_{k-1})^{4q}}\exp\Big(-\frac{\alpha_{2}}{4(\rho-1)\alpha_{3}^{2}(\Delta t_{k-1})^{\rho\varepsilon+\frac{m-1}{3}-\frac{\rho(m+1)\varepsilon}{2}}}\Big)\Big].
\end{eqnarray}
Therefore, combining \eqref{65} and \eqref{71} yields that for any $k\in\{1,2,...,n_{T}\}$,
\begin{eqnarray}
\!\!&&\!\!\mathbb{E}^{N}\Big[\bar{Z}_{t_{k}-}^{-2mq}\mathbf{1}_{\{\bar{Z}_{t_{k}-}\leq(\Delta t)^{\varepsilon}\bar{Z}_{t_{k-1}}\}}\Big]	\nonumber\\
\!\!&\leq&\!\!(\Delta t)^{2q}\Big(\mathbb{E}^{N}\!\Big[\bar{Z}_{t_{k}-}^{-4mq}\big(\frac{\Delta t_{k-1}}{\Delta t}\big)^{4q}\Big]\Big)^{\frac{1}{2}}\bigg(\mathbb{E}^{N}\Big[\bar{Z}_{t_{k-1}}^{\frac{8q}{(\rho-1)\varepsilon}}\Big]\nonumber\\
\!\!&&\!\!+\mathbb{E}^{N}\Big[\frac{1}{(\Delta t_{k-1})^{4q}}\exp\Big(-\frac{1}{8(\rho-1)^{2}\alpha_{3}^{2}(\Delta t_{k-1})^{\frac{1}{3}+\rho\varepsilon}}\Big)\Big]\nonumber\\
\!\!&&\!\!+\mathbb{E}^{N}\Big[\frac{1}{(\Delta t_{k-1})^{4q}}\exp\Big(-\frac{\alpha_{2}}{4(\rho-1)\alpha_{3}^{2}(\Delta t_{k-1})^{\rho\varepsilon+\frac{m-1}{3}-\frac{\rho(m+1)\varepsilon}{2}}}\Big)\Big]\bigg)^{\frac{1}{2}}.
\end{eqnarray}
Then we infer
\begin{eqnarray}\label{a96}
\!\!&&\!\!\Big\|\sup_{k=1,2,...,n_{T}}\mathbb{E}^{N}\big[\bar{Z}_{t_{k}-}^{-2mq}\mathbf{1}_{\{\bar{Z}_{t_{k}-}\leq(\Delta t)^{\varepsilon}\bar{Z}_{t_{k-1}}\}}\big]\Big\|_{L_{2}(\Omega, \mathbb{R})}\nonumber\\
\!\!&=&\!\!\Big(\mathbb{E}\Big[\sup_{k=1,2,...,n_{T}}\Big(\mathbb{E}^{N}\big[\bar{Z}_{t_{k}-}^{-2mq}\mathbf{1}_{\{\bar{Z}_{t_{k}-}\leq(\Delta t)^{\varepsilon}\bar{Z}_{t_{k-1}}\}}\big]\Big)^{2}\Big]\Big)^{\frac{1}{2}}\nonumber\\
\!\!&\leq&\!\!(\Delta t)^{2q}\bigg(\mathbb{E}\bigg[\sup_{k=1,2,...,n_{T}}\mathbb{E}^{N}\!\Big[\bar{Z}_{t_{k}-}^{-4mq}\big(\frac{\Delta t_{k-1}}{\Delta t}\big)^{4q}\Big]\Big(\mathbb{E}^{N}\Big[\bar{Z}_{t_{k-1}}^{\frac{8q}{(\rho-1)\varepsilon}}\Big]\nonumber\\
\!\!&&\!\!+\mathbb{E}^{N}\Big[\frac{1}{(\Delta t_{k-1})^{4q}}\exp\Big(-\frac{1}{8(\rho-1)^{2}\alpha_{3}^{2}(\Delta t_{k-1})^{\frac{1}{3}+\rho\varepsilon}}\Big)\Big]\nonumber\\
\!\!&&\!\!+\mathbb{E}^{N}\Big[\frac{1}{(\Delta t_{k-1})^{4q}}\exp\Big(-\frac{\alpha_{2}}{4(\rho-1)\alpha_{3}^{2}(\Delta t_{k-1})^{\rho\varepsilon+\frac{m-1}{3}-\frac{\rho(m+1)\varepsilon}{2}}}\Big)\Big]\Big)\bigg]\bigg)^{\frac{1}{2}}\nonumber.
\end{eqnarray}
In view of \eqref{0.39} and Corollary $\ref{cor3.2}$, we have
\begin{eqnarray}\label{c83}
\mathbb{E}\Big[\sup_{k=1,2,...,n_{T}}\bar{Z}_{t_{k}-}^{-8mq}\big(\frac{\Delta t_{k-1}}{\Delta t}\big)^{8q}\Big]<\infty.
\end{eqnarray}
By Theorem \ref{thm3.5} and Corollary \ref{cor3.2}, for any $q_{0}\geq1$, the triangle inequality leads to
\begin{eqnarray}\label{93}
\mathbb{E}\big[\sup_{k=0,1,...,n_{T}}\bar{Z}_{t_{k}}^{q_{0}}\big]\!\leq\! C\mathbb{E}\big[\sup_{k=0,1,...,n_{T}}\vert\bar{Z}_{t_{k}}\!-\!Z_{t_{k}}\vert^{q_{0}}\big]\!+\!C\mathbb{E}\big[\sup_{k=0,1,...,n_{T}}Z_{t_{k}}^{q_{0}}\big]\!<\!\infty.
\end{eqnarray}
Hence using the H\"{o}lder inequality, Conditional Jensen's inequality and the above results gives
\begin{eqnarray}\label{x99}
\!\!&&\!\!\Big\|\sup_{k=1,2,...,n_{T}}\mathbb{E}^{N}\big[\bar{Z}_{t_{k}-}^{-2mq}\mathbf{1}_{\{\bar{Z}_{t_{k}-}\leq(\Delta t)^{\varepsilon}\bar{Z}_{t_{k-1}}\}}\big]\Big\|_{L_{2}(\Omega, \mathbb{R})}\nonumber\\
\!\!&\leq&\!\! C(\Delta t)^{2q}\Big(\mathbb{E}\Big[\sup_{k=1,2,...,n_{T}}\!\mathbb{E}^{N}\big[\bar{Z}_{t_{k}-}^{-8mq}\big(\frac{\Delta t_{k\!-\!1}}{\Delta t}\big)^{8q}\big]\Big]\Big)^{\frac{1}{4}}\! \bigg[\Big(\mathbb{E}\Big[\sup_{k=1,2,...,n_{T}}\!\mathbb{E}^{N}\big[\bar{Z}_{t_{k\!-\!1}}^{\frac{16q}{(\rho\!-\!1)\varepsilon}}\big]\Big]\Big)^{\frac{1}{4}}\nonumber\\
\!\!&&\!\!+\Big(\mathbb{E}\Big[\sup_{k=1,2,...,n_{T}}\mathbb{E}^{N}\Big[\frac{1}{(\Delta t_{k-1})^{8q}}\exp\Big(-\frac{1}{4(\rho-1)^{2}\alpha_{3}^{2}(\Delta t_{k-1})^{\frac{1}{3}+\rho\varepsilon}}\Big)\Big]\Big]\Big)^{\frac{1}{4}}\nonumber\\
\!\!&&\!\!+\Big(\mathbb{E}\Big[\sup_{k=1,2,...,n_{T}}\!\!\mathbb{E}^{N}\Big[\frac{1}{(\Delta t_{k\!-\!1})^{8q}}\exp\Big(\!-\!\frac{\alpha_{2}}{2(\rho\!-\!1)\alpha_{3}^{2}(\Delta t_{k\!-\!1})^{\rho\varepsilon+\frac{m\!-\!1}{3}-\frac{\rho(m\!+\!1)\varepsilon}{2}}}\Big)\Big]\Big]\Big)^{\frac{1}{4}}\bigg]\nonumber\\ 
\!\!&\leq&\!\!C(\Delta t)^{2q}\Big(\mathbb{E}\Big[\sup_{k=1,2,...,n_{T}}\bar{Z}_{t_{k}-}^{-8mq}\big(\frac{\Delta t_{k-1}}{\Delta t}\big)^{8q}\Big]\Big)^{\frac{1}{4}}\bigg[\Big(\mathbb{E}\Big[\sup_{k=1,2,...,n_{T}}\bar{Z}_{t_{k-1}}^{\frac{16q}{(\rho-1)\varepsilon}}\Big]\Big)^{\frac{1}{4}}\nonumber\\
\!\!&&\!\!+\Big(\mathbb{E}\Big[\sup_{k=1,2,...,n_{T}}\frac{1}{(\Delta t_{k-1})^{8q}}\exp\Big(-\frac{1}{4(\rho-1)^{2}\alpha_{3}^{2}(\Delta t_{k-1})^{\frac{1}{3}+\rho\varepsilon}}\Big)\Big]\Big)^{\frac{1}{4}}\nonumber\\
\!\!&&\!\!+\Big(\mathbb{E}\Big[\sup_{k=1,2,...,n_{T}}\frac{1}{(\Delta t_{k-1})^{8q}}\exp\Big(-\frac{\alpha_{2}}{2(\rho-1)\alpha_{3}^{2}(\Delta t_{k-1})^{\rho\varepsilon+\frac{m-1}{3}-\frac{\rho(m+1)\varepsilon}{2}}}\Big)\Big]\Big)^{\frac{1}{4}}\bigg]\nonumber\\
\!\!&<&\!\!\infty,
\end{eqnarray}
where we have used in the last inequality the fact that for $p\geq 0$, the function $x\mapsto x^{p}e^{-x}$ is bounded by a constant $C_{p}$ for $x\in[0,\infty)$. The proof is completed.
\end{proof}

\section*{Disclosure statement}

The authors declare that they have no conflict of interest.

\section*{Data availability statement}
The datasets generated during and/or analyzed during the current study are available from the corresponding author on reasonable request.

\bibliographystyle{abbrv}

\end{document}